\documentclass[11pt]{amsart}
\usepackage{amssymb,amsmath, amsthm, mathrsfs, amscd, amssymb, enumerate}
\usepackage{hyperref}
\usepackage{cleveref,comment,xcolor}
\usepackage[margin=28mm]{geometry}

\hypersetup{
    colorlinks,
    linkcolor={red!50!black},
    citecolor={blue!50!black},
    urlcolor={blue!50!black}
}

\newtheorem{theorem}{Theorem}[section]
\newtheorem{lemma}[theorem]{Lemma}
\newtheorem{prop}[theorem]{Proposition}
\newtheorem{coro}[theorem]{Corollary}

\theoremstyle{definition}
\newtheorem{ex}{Example}

\theoremstyle{remark}
\newtheorem{remark}[theorem]{Remark}

\AddToHook{env/lemma/begin}{\crefalias{theorem}{lemma}} \AddToHook{env/prop/begin}{\crefalias{theorem}{prop}} \AddToHook{env/coro/begin}{\crefalias{theorem}{coro}} \AddToHook{env/conj/begin}{\crefalias{theorem}{conj}} \AddToHook{env/defn/begin}{\crefalias{theorem}{defn}} \AddToHook{env/remark/begin}{\crefalias{theorem}{remark}}

\usepackage{xcolor}
\usepackage{pgfplots}

\usepackage{subcaption}
\renewcommand{\hat}{\widehat}

\newcommand{\R}{\mathbb{R}}
\newcommand{\cP}{\mathcal{P}}

\newcommand{\N}{\mathbb{N}}

\newcommand{\Z}{\mathbb{Z}}

\newcommand{\T}{\mathbb{T}}
\newcommand{\U}{\mathbb{U}}
\newcommand{\bU}{\mathbf{U}}

\newcommand{\les}{\lesssim}
\newcommand{\less}{\lessapprox}
\newcommand{\eps}{\epsilon}

\newcommand{\inn}[1]{\langle  #1 \rangle}

\title{Weighted wave envelope estimates for the parabola}

\author[Jongchon Kim]{Jongchon Kim}
\address[Jongchon Kim]{Department of Mathematics, City University of Hong Kong, Hong Kong SAR}
\email{jongckim@cityu.edu.hk}

\author[Hyerim Ko]{Hyerim Ko}
\address[Hyerim Ko]{Department of Mathematics, and Institute of Pure and Applied Mathematics, Jeonbuk National University, Jeonju, 54896, Republic of Korea}
\email{kohr@jbnu.ac.kr}

\begin{document}

\begin{abstract}
In this paper, we extend the C\'ordoba--Fefferman square function estimate for the parabola to a weighted setting. Our weighted square function estimate is derived from a weighted wave envelope estimate for the parabola. The bounds are formulated in terms of families of multiscale tubes together with weight parameters that quantify the distribution of the weight. As  applications, we obtain  weighted $L^p$-estimates for a class of Fourier multiplier operators and for solutions to the free Schr\"{o}dinger equation.
\end{abstract}

\maketitle

\section{Introduction}
The paper is concerned with weighted square function estimates for the parabola and some of its applications. Let $\cP$ denote the truncated parabola 
\[\mathcal P=\{(t,t^2)\in \R^2: |t|\le 1\}\]
and $N_{R^{-1}} \cP = \{(t,t^2+\eta)\in \R^2: |t|\le 1,\; |\eta| \leq R^{-1} \}$ denote its $R^{-1}$-neighborhood for a large $R\geq 1$.  We consider the canonical covering of $N_{R^{-1}} \cP$ by finitely overlapping parallelograms $\theta$ of dimensions $R^{-1/2}\times R^{-1}$. Given a function $f$ whose Fourier transform is supported on $N_{R^{-1}} \cP$, we decompose $f=\sum_\theta f_\theta$, where $\hat f_\theta$ is supported on $\theta$. 
This can be done, for example, by using a smooth partition of unity subordinate to a covering of the interval $[-1,1]$ by finitely overlapping intervals of length $\sim R^{-1/2}$ (see e.g. the proof of \Cref{thm:BR}).

By Plancherel's theorem, these functions $\{ f_\theta\}$ are almost orthogonal in $L^2(\R^2)$: $\| f\|_{L^2}^2 \leq C \sum_{\theta} \| f_\theta \|_{L^2}^2$. Moreover, the family  exhibits certain $L^p$ orthogonality due to the curvature properties of the parabola for some $p$ larger than $2$. For instance, the classical square function estimate for the parabola (see \cite{Co}) states that 
\begin{equation}\label{eqn:Fef}
 \|f\|_{L^4(\R^2)} \leq C \big\| \big(\sum_\theta |f_\theta|^2 \big)^{1/2}\big\|_{L^4(\R^2)}.
\end{equation}
This inequality relies on the geometric observation by Fefferman \cite{Fe} that the algebraic sums $\theta+ \theta'$ overlap only finitely often as $\theta$ and $\theta'$ vary. See also \cite{GGPRY,Mal, GM2} for extensions to non-degenerate curves in higher dimensions. Square function estimates of the form \eqref{eqn:Fef} have several important applications in harmonic analysis. The sharp square function estimate \eqref{eqn:Fef} is known to imply sharp results for the Kakeya maximal function, the Bochner-Riesz multipliers,  the Fourier restriction operator, and local smoothing estimates for the Schr\"odinger equation; see \cite{Carbery, Yung} and references therein. For the paraboloid in higher dimensions, it is conjectured that \eqref{eqn:Fef} holds with $L^4(\R^2)$ replaced by $L^{\frac{2d}{d-1}}(\R^d)$, which remains wide open. 

 The main goal of this paper is to establish weighted square function estimates of the form
\[ 
	\|f\|_{L^p(Hdx)} \le C_{p,H}(R) \big\| \big(\sum_\theta |f_\theta|^2 \big)^{1/2}\big\|_{L^p(\R^2)}
\]
for $2\leq p\leq 4$ and to explore some of their consequences. Here and throughout the paper, a nonnegative locally integrable function   is called a weight. To describe the constant $C_{p,H}(R)$, we need to introduce a family of tubes originating from a multiscale analysis. Let $s$ be a dyadic number in the range $R^{-1/2}\le s\le1$. At each scale $s$, we cover $N_{s^2} \cP$  by canonical blocks $\{\tau \}$ of dimension $s\times s^{2}$ and use $|\tau|=s$ to denote the scale. For the smallest scale $s=R^{-1/2}$, these blocks are just $\{ \theta \}$. 

We fix a dyadic $s\in [R^{-1/2},1]$. For each $\tau$ with $|\tau|=s$, we consider a linear transform $L_\tau$ determined by the parabolic rescaling (see \eqref{eqn:Ltau}) for which $L_\tau ( [-\frac1 2,\frac 1 2]^2 )$ is a parallelepiped dual to $\tau$ of dimensions $s^{-1} \times s^{-2}$ and orthogonal to $\tau$. Let $\T_{\tau}$ denote the tiling of $\R^2$ by  translates of the dual parallelepiped: 
\begin{equation}\label{set of tube}
		\T_{\tau} = \{ L_\tau(z+q): z\in \Z^2 \}, \;\; q=[-1/2,1/2]^2.
\end{equation}
Next, we consider the tiling of $\R^2$ by the dilated family of tubes
\[
\U_{\tau} = \{ Rs^2 \cdot T : T\in \T_\tau \}.
\]
Each $U\in \U_{\tau}$ is thus a parallelepiped of dimensions $Rs \times R$. 
Let $\U$ denote the union of $\U_{\tau}$ for all $\tau$ ranging over all dyadic scales $R^{-1/2}\leq s \leq 1$. For a given $U\in \U$, we let $\tau(U)$ denote the $\tau$ such that $U\in \U_{\tau}$.

\begin{figure}\label{tubes}

\begin{tikzpicture}[scale=0.7]
	\draw (-9,4) -- (-7.5,4.5) -- (-7.5,5) -- (-9,4.5) -- cycle;
	\node at (-8.2, 4.5) {$\tau$};
	\node at (-8.2, 3.9) {$s$};
	\node at (-7.1, 4.9) {$s^{2}$};
	\draw (-4,6) -- (-5,6) -- (-6,9) -- (-5,9) -- cycle;
	\node at (-5, 7.5) {$T$};
	\node at (-4.3, 5.65) {$s^{-1}$};
	\node at (-3.9, 7.6) {$s^{-2}$};
	\draw (0.05,0) -- (-3,0) -- (-6,9) -- (-3,9) -- cycle;
	\node at (-2.8, 4.5) {$U$};
	\node at (-1.5, -0.4) {$Rs$};
	\node at (-1, 4.5) {$R$};
\end{tikzpicture}
	\caption{ $U\in \U_\tau$, $T\in \T_\tau$ such that $T\subset U$ for $|\tau|=s$.}
\end{figure}
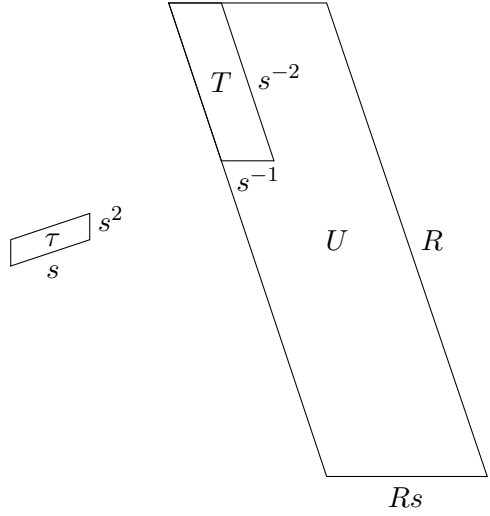

Given $U\in \U$, we define
\begin{align}\label{kappa_def}
\kappa_{p,H}(U) =   \max_{\substack{T \in \T_{\tau(U)}:\\ T\subset U}} \left(\frac{H(T)}{|T|}\right)^{\frac{1}{4}}\left(\frac{H(U)}{|U|}\right)^{\frac{1}{p}-\frac{1}{4}},
\end{align}
where we write $H(E) := \int_{E} H$ for a measurable set $E\subset \R^2$.
We are now ready to state our weighted square function estimates. 
\begin{theorem}\label{cor:sqweighted}
	Let $2\leq p\leq 4$ and $H$ be a weight. For any function $f$ whose Fourier transform is supported on $N_{R^{-1}} (\cP)$,  we have, 
	\begin{align}\label{sqThm}
		\|f\|_{L^p(Hd x)} \less  \sup_{U\in \U} \kappa_{p,H}(U)   \big\| \big(\sum_\theta |f_\theta|^2 \big)^{1/2}\big\|_{L^p(\R^2)}.
	\end{align}
\end{theorem}
Here we mean by $A \less B$  an inequality of the form $A \leq C_{\eps} R^\eps B$ for any $\eps>0$. We also use the notation $A \les B$ to denote $A\leq C B$ with an absolute constant  $C>0$,
possibly depending on parameters such as $p$ and $\alpha$.

\begin{remark}
  \Cref{cor:sqweighted} (and \Cref{thm:sqweighted} to be stated) remains valid when the parabola $\cP$ is replaced by a compact planar curve with curvature comparable to 1 for which the bilinear restriction estimate (cf. \Cref{thm:bil}) applies. In particular, it holds for any function whose Fourier transform is supported in a small neighborhood of the unit circle under a corresponding modification in the collections $\T_\tau$ and $\U_\tau$ outlined before the statement of \Cref{thm:BR} below. Moreover, the proofs of  \Cref{cor:sqweighted} and \Cref{thm:sqweighted} also extend to nonnegative locally finite Borel measures $\mu$ on $\R^2$, with $H(E)$ and $L^p(Hdx)$ replaced by $\mu(E)$ and $L^p(d\mu)$, respectively.
\end{remark}

For $H\equiv 1$, we have $\kappa_{p,H}(U)=1$ for any $U\in \U$. Thus, when $p=4$,  \Cref{cor:sqweighted} essentially recovers   the classical square function estimate \eqref{eqn:Fef}. For $2\leq p<4$, the $H\equiv 1$ case of \Cref{cor:sqweighted}  recovers  square function estimates due to Gan \cite{Gan}, where more general small cap square function estimates are established. The case $p=2$ of \Cref{cor:sqweighted} is essentially contained in \cite{CarberySeeger}; this was kindly pointed out to us by Tony Carbery (see also \cite[pages 17--19]{CarberySlide}).
Our weighted square function estimates are inspired by weighted decoupling inequalities for the paraboloids, which have been extensively studied in recent years and applied to problems such as the Falconer distance set conjecture and Bochner-Riesz means; see, e.g., \cite{GIOW, DORZ, GanWu, Kim} and references therein. 

We derive an upper bound for the constant $\sup_{U\in \U} \kappa_{p,H}(U)$ for $\alpha$-dimensional weights.
\begin{ex}[$\alpha$-dimensional weights]\label{ex_alpha}
	Let $0\leq \alpha \leq 2$. Suppose that a weight $H$ is  $\alpha$-dimensional in the sense that 
	\begin{align*}
		\inn{H}_\alpha:=\sup_{(z,\rho)\in \R^2 \times [1,\infty]} \rho^{-\alpha} H(B_\rho(z)) \les 1.
	\end{align*}
Here $B_\rho(z)$ denotes the ball of radius $\rho$ centered at $z$ (and we simply write $B_\rho$ when centered at the origin).
Then
	\begin{equation}\label{eqn:kappaAlpha}
		\sup_{U\in \U} \kappa_{p,H}(U)
	\;\les\;  R^{-(2-\alpha)\left(\frac{1}{p}-\frac{1}{4}\right)}.
	\end{equation}
To see this, let  $T\in \T_\tau$ and $U\in \U_\tau$ for some $|\tau|=s$. Since $T$ and $U$ are covered by $O(s^{-1})$ balls of radius $s^{-1}$ and $Rs$, respectively, we have 
\begin{align*}
  H(T) &\lesssim s^{-1} s^{-\alpha},\\
  H(U) &\lesssim s^{-1} (Rs)^{\alpha}.
\end{align*}
On the other hand, $|T|\sim s^{-3}$ and $|U|\sim R^2 s$. Therefore,
	\begin{equation*}
		\sup_{U\in \U} \kappa_{p,H}(U) 
		\;\les\; \max_{R^{-1/2}\le s\le1} (s^{2-\alpha})^{\frac 1 4} \big( (Rs)^{-(2-\alpha)}\big)^{\frac1 p - \frac 1 4},
	\end{equation*}
and the maximum is attained at the scale $s=1$, which yields \eqref{eqn:kappaAlpha}.
\end{ex}

\begin{ex}[Unit ball]\label{ex_ball}
We examine the sharpness of \Cref{cor:sqweighted} for the weight $H=1_{B_1}$.
Since $H$ is $\alpha$-dimensional for every $\alpha\in [0,2]$, \Cref{cor:sqweighted} together with \eqref{eqn:kappaAlpha} shows that
\begin{align*}
\| f\|_{L^p(B_1)} \less R^{-2(\frac 1 p-\frac 1 4)}  \big\| \big(\sum_\theta |f_\theta|^2 \big)^{1/2}\big\|_{L^p(\R^2)}.
\end{align*}

This estimate is essentially sharp for all $2\leq p \leq 4$. Indeed, let $\hat{f_\theta}$ be an $L^1$-normalized smooth bump function supported on $\theta$. In this case,
$|\sum_\theta f_\theta(x)| \gtrsim \# \{ \theta\}$ for $x\in B_c$ for a sufficiently small $c>0$ and $|f_\theta|$ decays rapidly away from the tube $\theta^*$ dual to $\theta$ centered at the origin, implying
\begin{align*} 
\|f\|_{L^p(B_1)} \gtrsim R^{\frac 1 2}  \; \text{ and } \; 
\big\| \big(\sum_\theta |f_\theta|^2 \big)^{1/2}\big\|_{L^p(\R^2)} \sim R^{\frac2 p}.
\end{align*}
\end{ex}

\begin{remark}\label{remark:lowerBoundSquare}
Fix $\alpha\in [1,2]$. Let $\sigma$ denote the infimum of exponents for which the bound
\[ 
\|f\|_{L^p(Hdx)} \less R^\sigma \big\| \big(\sum_\theta |f_\theta|^2 \big)^{\frac12} \big\|_{L^p(\R^2)}
\]
holds for any weight $H$ with $\langle H\rangle_\alpha\le1$.  \Cref{cor:sqweighted} and \eqref{eqn:kappaAlpha}  yield the upper bound 
\begin{equation}\label{eqn:upperSigma}
    \sigma \le -(2-\alpha)\left(\frac{1}{p}-\frac{1}{4}\right).
\end{equation}
When $p=2$ or $p=4$, the upper bound matches with the lower bound  
\begin{equation}\label{eqn:lowerSigma}
   \sigma \ge \max\Big(-2\big(\frac1p-\frac14\big),~-\frac{2-\alpha}{2p},~-(2-\alpha)\big(\frac{1}{p}-\frac{1}{6}\big) \Big).  
\end{equation}
The first lower bound follows from \Cref{ex_ball}. For the second lower bound, fix $\theta$ and take $f=f_\theta$ as in \Cref{ex_ball}, and set $H=R^{\frac{\alpha-2}2} 1_{\theta^*}$. A direct computation shows $\inn{H}_{\alpha}\les 1$, which yields the second lower bound. The third lower bound can be obtained by using a special solution to Schr\"odinger equation studied by Barcel\'o, Bennett, Carbery, Ruiz, and Vilela \cite{BBCRV0}; see \Cref{sec:BBCRV}.

On the other hand, there are gaps between \eqref{eqn:upperSigma} and \eqref{eqn:lowerSigma} for intermediate $2<p<4$. Nevertheless, we will show that, for each $\alpha \in (1,2)$, there exists an $\alpha$-dimensional weight for which \Cref{cor:sqweighted} gives sharp $L^p$ weighted square function estimates when $2\leq p\leq 4/(3-\alpha)$ or $p=4$; see \Cref{sec:exInterpolation}. 
\end{remark}

We present three consequences of \Cref{cor:sqweighted} in \Cref{sec:app}.\\[1ex]
$(i)$ Weighted $L^p$ bounds for Fourier multipliers supported on a small neighborhood of the unit circle.\\[1ex]
$(ii)$ Weighted and frequency–localized $L^p$ bounds for the one–dimensional Schr\"odinger propagator.\\[1ex]
$(iii)$ Local smoothing estimates for the Schr\"{o}dinger equation with respect to fractal measures 
satisfying parabolic or Euclidean ball conditions. \\[1ex]
In cases $(i)$ and $(ii)$, the dependence on the weight is quantified by $\max_{U\in\U}\kappa_{p,H}(U)$, which extends classical unweighted estimates.

\Cref{cor:sqweighted} is a consequence of a weighted $L^p$ wave envelope estimate for the parabola.
\begin{theorem}\label{thm:sqweighted}
	Let $2\leq p\leq 4$ and $H$ be a weight. If $\widehat f$ is supported on $N_{R^{-1}} (\cP)$, we have  
\begin{equation}\label{eqn:envelope}
	\| f \|_{L^p(Hd x)}^p \less  \sum_{R^{-1/2} \leq s\leq 1} \sum_{|\tau|=s} \sum_{U\in \U_{\tau}} \kappa_{p,H}(U)^p |U|^{1-\frac{p}{2}} \big\| (\sum_{\theta \subset \tau} |f_{\theta}|^2 )^{1/2}  \big\|_{L^2(w_{U})}^p.
\end{equation}    
\end{theorem}
Here $w_U$ denotes an $L^\infty$-normalized weight which decays rapidly away from $U$. 

The estimate \eqref{eqn:envelope} is sharp for the unit ball example in \Cref{ex_ball}, where the term $s=1$ on the right-hand side of \eqref{eqn:envelope} dominates. We present an example where $\|f\|_{L^p(Hdx)}^p$ essentially matches the contribution from $s = R^{-1/2}$.

\begin{ex}
Let $Y\subset \R^2$ and $H=1_Y$. Consider $f=\sum_{\theta} f_\theta$ such that $\{f_\theta\}$ have essentially disjoint supports on $Y$. Then
\begin{align*}
	\| f\|_{L^p(Y)}^p \sim \sum_{\theta} \| f_\theta \|_{L^p(Y)}^p = \sum_{\theta} \sum_{U\in \U_\theta} \| f_\theta\|_{L^p(U\cap Y)}^p.
\end{align*}
We further assume that $|f_\theta|$ is  essentially constant on each $U\in \U_\theta$, which is natural in view of the uncertainty principle. Then
\begin{equation}\label{eqn:exKnapp}
	\| f_\theta\|_{L^p(U\cap Y)}^p 
	\sim \frac{|U\cap Y|}{|U|} \| f_\theta\|_{L^p(U)}^p 
	\sim \frac{|U\cap Y|}{|U|} |U|^{1-\frac{p}{2}} \| f_\theta\|_{L^2(U)}^p.
\end{equation}
Hence,
\[ 
\|f\|_{L^p(Y)}^p \;\sim\;
\sum_{|\theta|=R^{-1/2}} \sum_{U\in \U_\theta}
\frac{|U\cap Y|}{|U|}\,|U|^{1-\frac{p}{2}}\,
\big\|f_\theta\big\|_{L^2(U)}^{p}.
\]

We note that $\U_{\theta} = \T_\theta$ forms an identical tiling of $\R^2$ by parallelepipeds of dimensions $R^{1/2}\times R$. Therefore, 
\begin{align}\label{eqn:22}
	\kappa_{p,H}(U)^p =  \frac{H(U)}{|U|} = \frac{|U\cap Y|}{|U|}, \;\; U\in \U_\theta.
\end{align}
Thus, $\|f \|_{L^p(Y)}^p$ is comparable to the $s=R^{-1/2}$ term on the right-hand side of \eqref{eqn:envelope}.
\end{ex}

Wave envelope estimates, namely estimates of the form \eqref{eqn:envelope} with $H\equiv 1$, were first developed in the breakthrough work of Guth, Wang, and Zhang \cite{GWZ} for the cone 
\[
\Gamma=\{\xi_1^2+\xi_2^2=\xi_3^2,~1/2\le \xi_3\le2\} \subset \R^3.
\]
Their wave envelope estimate for $\Gamma$ implies, among other consequences, the sharp $L^4$ square function estimate for $\Gamma$ and the sharp local smoothing estimate for the wave equation in 2+1 dimension.
The $p=4$ and $H\equiv 1$ case of \Cref{thm:sqweighted} recovers the $L^4$ wave envelope estimate \cite[Equation (8)]{GM1}, which is  implicit in \cite{GWZ}. 
We note that \cite{GM1} established more refined versions of the $L^4$ wave envelope estimates, termed amplitude-dependent wave envelope estimates, for both the parabola $\cP$ and the cone $\Gamma$.

For the proof of \Cref{thm:sqweighted}, the classical approach used to establish the square function estimate \eqref{eqn:Fef} is not applicable, as it relies critically on the even exponent $4$ and Plancherel’s theorem, neither of which extend to weighted settings or general exponents.  Instead, we adopt a more robust strategy used by Gan \cite{Gan} for proving small cap square function estimates. This method employs a multiscale bilinear reduction argument from \cite{BG, DGW}, together with the bilinear restriction theorem (see e.g. \cite{Tao}). One of the main contributions of the present paper is an extension of the method to the weighted setting that effectively exploits the presence of the weights without imposing any additional assumptions.

\subsection*{Organization of the paper}
In \Cref{sec:app}, we deduce \Cref{cor:sqweighted} from \Cref{thm:sqweighted} and discuss applications of \Cref{cor:sqweighted}. 
In \Cref{sec:MoreExamples}, we present additional examples related to the sharpness of \Cref{cor:sqweighted} for $\alpha$-dimensional weights.
In \Cref{sec:bil}, we establish \Cref{thm:sqweighted} using the bilinear restriction theorem. 
In \Cref{sec:ls}, we prove fractal local smoothing estimates to be stated in \Cref{sec:app}. 
In \Cref{examp}, we look at examples and derive necessary conditions for  \Cref{cor:SchFreqLocal}, \Cref{FLSS_par} and \Cref{FLSS}.
Finally, in \Cref{sec:pt}, we give a proof of \Cref{ptwise bound}, a multiscale broad-narrow decomposition.

\subsubsection*{Notations}
We summarize here the notations that will be used frequently throughout the paper.\\[3pt]
$\bullet$ We write $A \less B$ to denote an inequality of the form $A \leq C_{\eps} R^{\eps} B$ for any $\eps>0$ where $C_{\eps}$ is an absolute constant.\\[3pt]
$\bullet$ We denote by $B_\rho(z)$ the ball of radius $\rho$ centered at $z\in \R^2$ and we simply write $B_\rho$ when centered at the origin.\\[3pt]
$\bullet$ For a Borel measure $\mu$ on $\R^2$, we set $\displaystyle{\inn{\mu}_\alpha:=\sup_{(z,\rho)\in \R^2\times [1,\infty) } \rho^{-\alpha} \mu(B_\rho(z))}$ and define $[\mu]_\alpha$ similarly except that the supremum is taken over $\rho>0$. Analogous conventions apply for other related quantities.

\section{Proof of \Cref{cor:sqweighted} and some applications}\label{sec:app}
We begin by deriving \Cref{cor:sqweighted} from the weighted envelope estimate in \Cref{thm:sqweighted}. Then we turn to some applications of \Cref{cor:sqweighted}. 

\subsection{Weighted wave envelope estimates imply weighted square function estimates}
In this section, we prove that \Cref{thm:sqweighted} implies \Cref{cor:sqweighted}.

Let $2\leq p\leq 4$. 
By H\"{o}lder's inequality, \Cref{thm:sqweighted} yields
\begin{equation*}
	\| f \|_{L^p(Hdx)}^p \less \sum_{R^{-1/2} \leq s\leq 1} \sum_{|\tau|=s} \sum_{U\in \U_{\tau}} \kappa_{p,H}(U)^p  \big\| \big(\sum_{\theta \subset \tau} |f_{\theta}|^2 \big)^{1/2}  \big\|_{L^p(w_{U})}^p.
\end{equation*}

After dominating $\kappa_{p,H}(U)^p$ by $\sup_{U\in \U} \kappa_{p,H}(U)^p$, we sum over all $U$. This yields, for each $s$, 
    \begin{align*}
        \sum_{|\tau|=s} \sum_{U\in \U_{\tau}} \big\| \big(\sum_{\theta \subset \tau} |f_{\theta}|^2 \big)^{1/2}  \big\|_{L^p(w_{U})}^p \les \sum_{|\tau|=s} \big\| \big(\sum_{\theta \subset \tau} |f_{\theta}|^2 \big)^{1/2}  \big\|_{L^p}^p \leq  \big\| \big(\sum_{\theta} |f_{\theta}|^2 \big)^{1/2}  \big\|_{L^p}^p.
    \end{align*}
For the last inequality, we use embedding $\ell^2\subset \ell^p$ for $p\ge2$.
Since $s$ ranges over dyadic numbers in $[R^{-\frac12},1]$, this gives \Cref{cor:sqweighted}.
\qed

\subsection{Weighted estimates for a radial Fourier multiplier}
Let $\psi$ be a smooth bump function supported on $[-1,1]$. 
We consider the Fourier multiplier transformation $S_R$  defined by 
\[ \widehat{S_R f} (\xi) = \psi( R (1-|\xi|) ) \hat{f}(\xi), \]
which plays a critical role in the theory of Bochner-Riesz means. It is well-known that 
\begin{equation}\label{eqn:BR}
	\| S_R f\|_{L^p(\R^2)} \less \| f\|_{L^p(\R^2)}, \;\; 2\leq p\leq 4,
\end{equation}
which follows from an interpolation of the trivial $L^2$-bound and the sharp $L^4$-bound due to C\'ordoba \cite{Co} which relies on the square function estimate \eqref{eqn:Fef} and  bounds for the Nikodym maximal function.

We present a weighted version of \eqref{eqn:BR}. For each dyadic scale $R^{-1/2}\leq s\leq 1$, we cover $N_{R^{-1}}\mathbb{S}^1$ by finitely overlapping rectangles $\tau$ of dimensions $s\times s^2$ and define $\T_\tau$, $\U_\tau$ and $\U$, accordingly. With this minor modification in mind, we obtain the following. 
\begin{theorem}\label{thm:BR}
Let $2\leq p\leq 4$ and $H$ be a weight. Then 
\[ 
			\| S_R f\|_{L^p(Hd x)} \less  \sup_{U\in \U} \kappa_{p,H}(U)     \| f\|_{L^p(\R^2)} .
\]	
\end{theorem}
\begin{proof}
The proof follows C\'{o}rdoba's argument for \eqref{eqn:BR} in \cite{Co}, with \Cref{cor:sqweighted} as the weighted square function input, so we only  sketch the argument. We divide $\R^2$ into four sectors by lines $y=\pm x$. Without loss of generality, we may replace $S_R$ by a smooth frequency projection to the part of $N_{R^{-1}}\mathbb{S}^1$ contained in one of the four sectors which includes the point $(0,-1)$.  
	
	Next, we cover $[-1,1]$ by finitely overlapping intervals $I$ of length $\sim R^{-1/2}$, and let $\{ \chi_I\}$ be a smooth partition of unity adapted to this covering. Then we have $S_R f=\sum_{I} S_R f_I$, where $\hat{f_I}(\xi_1,\xi_2) =  \chi_I(\xi_1) \hat{f}(\xi_1,\xi_2)$. By \Cref{cor:sqweighted}, we have
	\[ \|S_R f\|_{L^p(Hd x)} \less    \max_{U\in \U} \kappa_{p,H}(U)   \big\| \big(\sum_I |S_R f_I|^2 \big)^{1/2}\big\|_{L^p(\R^2)}. \]
	By duality and the boundedness of the Nikodym maximal function, for any $2\leq p\leq 4$,
    \[ \big\| \big(\sum_I |S_R f_I|^2 \big)^{1/2}\big\|_{L^p(\R^2)} \less \big\| \big(\sum_I |f_I|^2 \big)^{1/2}\big\|_{L^p(\R^2)}.	\]
   Finally, by the Littlewood-Paley inequality for equally spaced intervals, we have  
	\begin{align}\label{LP}
		\big\| \big(\sum_I |f_I|^2 \big)^{1/2}\big\|_{L^p(\R^2)} \les \| f\|_{L^p(\R^2)}, \;\;\; p \geq 2,
	\end{align}
	which completes the proof.
\end{proof}

\subsection{Weighted estimates for the Schr\"odinger equation}\label{sec:etaDefn}
Let \[
e^{-it\partial_x^2}f(x)
=(2\pi)^{-1}\int e^{ix\xi}e^{it\xi^2}
\widehat f(\xi)\,d\xi
\]
denote the solution to the free Schr\"odinger equation
\begin{align*}
	\begin{cases}
		i\partial_tu = \partial_x^2 u, &(x,t)\in \R \times \R \\
		u(x,0)=f(x), &x\in \R.    
	\end{cases}
\end{align*}

Let $\eta$ be a Schwartz function such that $\widehat{\eta}$  is compactly supported on $[-1,1]$ and $|\eta(t)|\sim 1$ on $[-1,1]$. We define the operator $\mathbf{U}_R$ by 
\begin{equation}\label{eqn:defnU_R}
	\mathbf{U}_R f(x,t)=\eta(R^{-1}t)e^{-it\partial_x^2}f(x).	
\end{equation}
If $\hat{f}$ is supported on $[-1,1]$, then $\hat{\mathbf {U}_R f}$ is supported on $N_{R^{-1}} \cP$. 

As another consequence of \Cref{cor:sqweighted}, we establish a weighted  estimate for the Schr\"odinger propagator. 
\begin{theorem}\label{thm:SchFreqLocal}
	Let $2\leq p\leq 4$ and $H$ be a weight. For any function $f$ whose Fourier transform is supported on $[-1,1]$, we have 
	\[ 
	\| \mathbf{U}_R f \|_{L^p(\R^2, Hdxdt)} \less   \sup_{U\in \U} \kappa_{p,H}(U)  R^{\frac1p} \| f\|_{L^p(\R)}.
	\]	
\end{theorem}

\begin{proof}
Let $\{\chi_I\}$ be the smooth partition of unity given in \Cref{thm:BR}. We decompose $f=\sum_{I} f_I$, where $\hat{f_I} =  \chi_I \hat{f}$. By \Cref{cor:sqweighted}, we have 
	\[ 
	\| \mathbf{U}_Rf \|_{L^p(\R^2, Hdxdt)}  \less   \max_{U\in \U} \kappa_{p,H}(U)   \big\| \big(\sum_I |\bU_R f_I|^2 \big)^{1/2}\big\|_{L^p(\R^2)}.
	\]	
	Thus, it suffices to verify that
	\begin{align}\label{sqfc}
		\Big\| \big(\sum_I | \mathbf{U}_R f_I|^2\big)^{1/2} \Big\|_{L^p(\R^2)}
		\less R^{\frac1p}\|f\|_{L^p(\R)}, \;\; 2\leq p\leq 4.
	\end{align}
	A detailed proof of \eqref{sqfc} can be found in a note by Yung \cite[Proof of Theorem 2]{Yung}. It follows from a standard duality argument similar to the one used in \cite{Co} and  \cite{MSS}. For completeness, we sketch the argument in \Cref{sec:sqfc}.
\end{proof}

As a corollary, we state a special case for $\alpha$-dimensional measures.
For $0\leq \alpha\leq 2$ and a measure $\mu$ defined on $\R^2$, we set
    \begin{align}\label{mu_ball}
    \inn{\mu}_\alpha := \sup_{(z,\rho)\in \R^2\times [1,\infty) } \rho^{-\alpha} \mu(B_\rho(z)). \end{align}
\begin{coro}\label{cor:SchFreqLocal}
	 Let $2\leq p\leq 4$ and $0\leq \alpha\leq 2$. 
     For any function $f$ whose Fourier transform is supported on $[-1,1]$ and $\mu$ satisfying $\inn{\mu}_\alpha \le1$, we have 
     \begin{align*}
\big\| \mathbf{U}_R f \big\|_{L^p(\R\times[0,R], \mu)} \less  R^{\frac1p-(2-\alpha)\big(\frac1p-\frac14\big)} \| f\|_{L^p(\R)}.
\end{align*}
\end{coro}
\begin{proof}
Recall \eqref{eqn:defnU_R}.
    By the Fourier localization property of $\bU_R f$, we may write $\bU_R f = \bU_R f * \varphi$ for a Schwartz function $\varphi \in \mathcal S(\R^2)$. 
    Consequently, by H\"{o}lder's inequality,
\[
\big|\mathbf U_Rf \ast \varphi \big|^p
\lesssim |\mathbf U_Rf|^p \ast |\varphi|
\]
for $p\ge1$. It follows that
\[
\int \big|\mathbf U_Rf(x,t)\big|^p\,d\mu(x,t)
\lesssim
\int \big|\mathbf U_Rf(x,t)\big|^p H(x,t)\,dxdt
\]
where 
$
   H=\mu \ast |\varphi|. 
$

We check that $H$ is $\alpha$-dimensional, using the dyadic decomposition $|\varphi|\les \sum_{j\in \N} 2^{-10j} 1_{B_{2^j}}$.  By using the decay and the assumption that $\inn{\mu}_{\alpha}\les 1$, we have
\begin{align*}
    \int_{B_\rho(z)} H  \les \sum_{j\in \N} 2^{-8j} \mu( B_{\rho}(z) + B_{2^j} ) \les \rho^{\alpha}.
\end{align*}
Thus, $\inn{H}_\alpha\les 1$.  Consequently, \Cref{thm:SchFreqLocal} and \eqref{eqn:kappaAlpha} give the desired estimate. 
\end{proof}

\Cref{cor:SchFreqLocal} can be obtained by interpolation between known $L^2$ and $L^4$ estimates. Indeed, when $p=2$, \Cref{cor:SchFreqLocal} recovers a bound due to Du and Zhang \cite{DuZhang} (see also \cite{Wolff2, Erdogan}). The $p=4$ case of \Cref{cor:SchFreqLocal} can be deduced from a local smoothing estimate for the Schr\"odinger equation:
\begin{equation}\label{eqn:Rogers}
	\| e^{-it\partial_x^2} f \|_{L^p(\R\times [0,1])} \leq C_{p,\gamma} \| f\|_{L^p_\gamma(\R)},\;\; p\in (2,\infty) \; \text{and} \; \gamma>\max\Big(0,1-\frac{4}{p}\Big),
\end{equation}
which is due to  Rogers \cite{Rogers}.  Here $L^p_\gamma(\R)$ denotes the $L^p$-Sobolev space equipped with the norm $\|f\|_{L_\gamma^p(\R)}=\|(1-\Delta)^{\gamma/2}f\|_{L^p(\R)}$. The regularity assumption on $\gamma$ in \eqref{eqn:Rogers} is essentially sharp.

The bound obtained in \Cref{cor:SchFreqLocal} is essentially sharp when $p=2$, provided that $\alpha\in [1,2]$, or $p=4$. Indeed, if
\begin{align}\label{global_LS}
\big\| \mathbf{U}_R f \big\|_{L^p(\R\times[0,R], \mu)} \les  R^{\zeta} \| f\|_{L^p(\R)}
\end{align}
holds, then 
\begin{align*}
\zeta \ge \begin{cases}
	\max\big(\frac12-\frac1p, ~\frac{\alpha}{2p}\big), \;\;&\alpha\in[1,2],\\[3pt]
	\max\big(\frac12-\frac1p, ~\frac{2\alpha-1}{2p}\big),\;\;&\alpha\in[0,1].
\end{cases}
\end{align*}
We discuss the details in \Cref{sec:global_low}.

\subsection{Fractal local smoothing estimates relative to parabolic balls}
The estimate \eqref{eqn:Rogers} can be regarded as an analogue of the local smoothing phenomenon for the wave equation, first discovered by Sogge \cite{Sogge}. Indeed, comparing \eqref{eqn:Rogers} with the sharp fixed-time estimate due to Miyachi \cite{Miyachi},
\[  \| e^{-it\partial_x^2}f \|_{L^p(\R)} \les_{p,\gamma} \| f\|_{L^p_\gamma(\R)}, \;\; p\in (1,\infty) \; \text{and} \; \gamma\geq  \Big|1-\frac{2}{p}\Big|,\]
it follows that averaging over a compact time interval yields a gain of $2/p$ derivatives whenever $p>4$. In the context of the Schr\"odinger equation, local smoothing estimates generally refer to such derivative gains obtained by averaging over a compact space-time region (see e.g. \cite{Sjolin}). 

Rogers \cite{Rogers} proved \eqref{eqn:Rogers} by connecting it to the Fourier restriction estimate for the parabola. See \cite{Yung} for a proof of  \eqref{eqn:Rogers} which relies on the square function estimate \eqref{eqn:Fef}. 

We seek to extend estimates of the form \eqref{eqn:Rogers} to general measures on $\R\times [0,1]$  that satisfy suitable size conditions. We refer to these as \textit{fractal local smoothing estimates} for the Schr\"odinger equation. 

For $0\leq \beta \leq 3$, we consider a class of Borel measures on $\R^2$ for which 
\begin{align*}
[\mu]_{\beta,\mathrm{par}} := \sup_{z\in \R^2, \rho>0}  \rho^{-\beta}\mu\big(B_{\rho,\mathrm{par}}(z)\big) \lesssim 1,
\end{align*} 
where $B_{\rho,\mathrm{par}}(z)$ denotes the parabolic ``ball" $(z_1-\rho,z_1+\rho)\times (z_2-\rho^2,z_2+\rho^2)$ for $z=(z_1,z_2)$.  
This class of measures naturally arises in view of the parabolic rescaling associated with the Schr\"odinger equation. For such measures, we consider the estimate
\begin{align}\label{fls_par}
\big\| e^{-it\partial_x^2}f \big\|_{L^p(\R\times [0,1],\mu)}
\le C [\mu]_{\beta,\mathrm{par}}^{1/p} \|f\|_{L_\gamma^p(\R)}.
\end{align}

\begin{theorem}[Parabolic $\beta$-dimensional case]\label{FLSS_par}
Let $0\le \beta \le 3$ and let $\mu$ be a Borel measure on $\R^2$ with $[\mu]_{\beta,\mathrm{par}}\le1$. 
\vspace{-.2cm}
\begin{enumerate}[(i)]
\item 
\textit{(Sufficiency)} 
For $2\le p\le 4$, there exists $C=C_{\beta,p,\gamma}>0$ such that \eqref{fls_par} holds whenever
\[
\gamma>\gamma_{\mathrm{par}}(\beta):=
  \begin{cases}
\frac{3-\beta}{4}, & \beta\in[1,3],\\[.4ex]
\frac{2-\beta}{2} - \frac{1-\beta}{p}, & \beta\in[0,1].
  \end{cases}
\]
\item 
\textit{(Necessity)} 
Conversely, if \eqref{fls_par} holds for some $0\le\beta\le 3$, then
\[
  \gamma\ge
  \begin{cases}
  \max\!\big\{1-\frac{\beta+1}{p},\, \frac{3-\beta}{2p}\big\}, & \beta\in[1,3],\\[1ex]
\max\!\big\{1-\frac{\beta+1}{p},\,\frac\beta p\}, & \beta\in[0,1].
  \end{cases}
\]
\end{enumerate}
\end{theorem}

The necessary condition shows that $\gamma> \gamma_{\mathrm{par}}(\beta)$ is essentially sharp for $p=4$ for all $0\leq \beta\leq 3$, and for $p=2$ for all $1\leq \beta\leq 3$.
We prove the sufficiency part of \Cref{FLSS_par} in \Cref{sec:ls} and the necessity part in \Cref{examp}.

\begin{figure}[!htbp]
	\centering
	\vspace*{8mm}
	\hfill
	\begin{subfigure}[b][2cm][c]{0.46\textwidth}
		\centering
		\begin{tikzpicture}[xscale=11,yscale=5]
			\pgfmathsetmacro{\betaR}{2.0}
			\pgfmathsetmacro{\ySufR}{(3 - \betaR)/4}
			\pgfmathsetmacro{\xkR}{2/(\betaR + 5)}
			\pgfmathsetmacro{\ykR}{(3 - \betaR)/(\betaR + 5)}
			\pgfmathsetmacro{\yMax}{\ySufR + 0.12}
			
			\begin{scope}
				\clip (0,0) rectangle (0.56,\yMax);
				
				\draw[->] (0,0) -- (0,\yMax);
				\draw[->] (0,0) -- (0.56,0);
				
				\draw[gray,densely dotted] (1/4,0) -- (1/4,\yMax);
				\draw[gray,densely dotted] (1/2,0) -- (1/2,\yMax);
				
				\draw[very thick]
				(1/4,\ySufR) -- (1/2,\ySufR);
				
				\draw[dotted,thick,domain=1/4:\xkR]
				plot (\x,{1 - (\betaR + 1)*\x});
				\draw[dotted,thick,domain=\xkR:1/2]
				plot (\x,{0.5*(3 - \betaR)*\x});
				
				\draw[gray,densely dotted]
				(\xkR,0) -- (\xkR,\ykR);
			\end{scope}

			\node[right]      at (0.56,0) {$\tfrac{1}{p}$};
			\node[below left] at (0,\yMax) {$\gamma$};
			\node[below,xshift=-1pt]      at (1/4,0) {$\tfrac{1}{4}$};
			\node[below]      at (1/2,0) {$\tfrac{1}{2}$};
			\node[below,xshift=2pt]      at (\xkR,0) {$\tfrac{2}{\beta+5}$};
		\end{tikzpicture}
		\vspace{-0.2cm}
		\caption{$1\le \beta \le 3$}
	\end{subfigure}
	\hfill
	\begin{subfigure}[b][2cm][c]{0.46\textwidth}
		\centering
		\begin{tikzpicture}[xscale=11,yscale=2.5]
			\pgfmathsetmacro{\betaL}{0.6}
			\pgfmathsetmacro{\ySufL}{(3 - \betaL)/4}
			\pgfmathsetmacro{\xkL}{1/(2*\betaL + 1)}
			\pgfmathsetmacro{\ykL}{\betaL/(2*\betaL + 1)}
			\pgfmathsetmacro{\yMax}{\ySufL + 0.12}
			
			\begin{scope}
				\clip (0,0) rectangle (0.56,\yMax);
				
				\draw[->] (0,0) -- (0,\yMax);
				\draw[->] (0,0) -- (0.56,0);
				
				\draw[gray,densely dotted] (1/4,0) -- (1/4,\yMax);
				\draw[gray,densely dotted] (1/2,0) -- (1/2,\yMax);
				
				\draw[very thick,domain=1/4:1/2]
				plot (\x,{0.5*(2 - \betaL) - (1 - \betaL)*\x});
				
				\draw[dotted,thick,domain=1/4:\xkL]
				plot (\x,{1 - (\betaL + 1)*\x});
				\draw[dotted,thick,domain=\xkL:1/2]
				plot (\x,{\betaL*\x});
				
				\draw[gray,densely dotted]
				(\xkL,0) -- (\xkL,\ykL);
			\end{scope}
			
			\node[right]      at (0.56,0) {$\tfrac{1}{p}$};
			\node[below left] at (0,\yMax) {$\gamma$};
			\node[below]      at (1/4,0) {$\tfrac{1}{4}$};
			\node[below,xshift=1pt]      at (1/2,0) {$\tfrac{1}{2}$};
			\node[below,xshift=-2pt]      at (\xkL,0) {$\tfrac{1}{2\beta+1}$};
		\end{tikzpicture}
		\vspace{-0.2cm}
		\caption{$0\le \beta \le 1$}
	\end{subfigure}
	
	\vspace{6mm}
	\caption{Sufficient (solid) and necessary (dotted) thresholds for
		\Cref{FLSS_par}.}
	\label{fig:par_two_panels}
\end{figure}
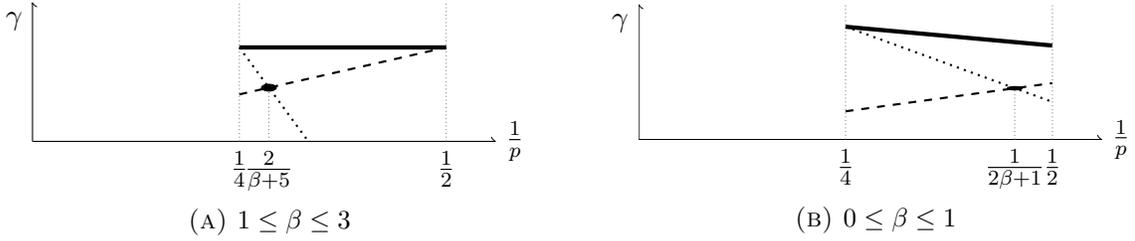

We compare the $p=2$ case of \Cref{FLSS_par} with known \textit{weighted Strichartz estimates}. For this purpose, we consider  parabolic Morrey-Campanato type classes. Given $\delta>0$ and $1\le q \le 3/\delta$, we define $\mathfrak L_{\mathrm{par}}^{\delta,q}$ to be the set of nonnegative weights $H\in L_{\text{loc}}^q(\R\times\R)$, equipped with the norm
\[
\|H\|_{\mathfrak L_{\mathrm{par}}^{\delta,q}}
:=
\sup_{(x,t)\in \R^{1+1}, r>0}
r^\delta \Big( \frac{1}{r^{3}}
\int_{B_{r, \mathrm{par}}(x,t)} H(y,s)^q\,dyds\Big)^{1/q}.
\]
At the endpoint $q=3/\delta$, $\mathfrak L_{\mathrm{par}}^{\delta,q}$ coincides with $L^{3/\delta}(\R^2)$. For $1\leq q<3/\delta$, $L^{3/\delta,\infty}(\R^2)  \hookrightarrow \mathfrak L_{\mathrm{par}}^{\delta,q}$ by H\"older's inequality for Lorentz spaces. Moreover, for $\mu=Hdxdt$,
\[[\mu]_{\beta,\mathrm{par}}= \|H\|_{\mathfrak L_{\mathrm{par}}^{3-\beta,1}}.\]

Barcel\'o et al. \cite{BBCRV} established weighted Strichartz estimates of the form
\begin{align}\label{MC-estimate}
\| e^{-it\partial_x^2}f\|_{L_{x,t}^2(H(x,t))} \le C \|H\|_{\mathfrak L_{\mathrm{par}}^{2\gamma+2,q}}^{1/2}\|f\|_{\dot{H}^\gamma}
\end{align}
where the exponent $2\gamma+2$ is determined by the scaling invariance. In \cite{BBCRV}, it was shown that \eqref{MC-estimate} holds for $\frac14 \le \gamma< \frac12$ and $1<q\le \frac{3}{2\gamma+2}$ (together with a higher-dimensional analogue). 
See \cite{BBCRV, KS} for the case $0\le \gamma <\frac14$ and \cite{BBCRV2} for results with time-dependent weights $H$.

Note that when $0< \beta\leq 1/2$, \eqref{MC-estimate} yields 
\begin{equation}\label{MC-estimate2}
   \| e^{-it\partial_x^2}f\|_{L_{x,t}^2(H(x,t))} \le C \| H\|_{\mathfrak L_{\mathrm{par}}^{3-\beta,q}}^{1/2}\|f\|_{\dot{H}^\gamma} 
\end{equation}
for $\gamma = (1-\beta)/2$ and $1<q\leq \frac{3}{3-\beta}$.
Since $q>1$, \eqref{MC-estimate2} does not seem to  imply estimates \eqref{fls_par} for measures $\mu=Hdxdt$. Nevertheless, for the  range $0<\beta\leq 1/2$, \eqref{MC-estimate2} is stronger than \Cref{FLSS_par}  in the sense that it provides a global estimate in both space and time and the regularity index $\gamma$ is optimal, matching the necessary condition in \Cref{FLSS_par} $(ii)$.

\subsection{Fractal local smoothing estimates}
We now consider fractal local smoothing estimates for the Schr\"odinger operator with $\alpha$-dimensional measures in the standard (non-parabolic) sense. Let $0\le \alpha \le 2$ and consider a class of Borel measures on $\R^2$ for which 
\begin{align}\label{mu-alpha}
[\mu]_{\alpha} := \sup_{z\in \R^2,\; \rho>0}  \rho^{-\alpha}\mu\big(B_\rho(z)\big) <\infty.
\end{align}
Here, the condition \eqref{mu-alpha} differs slightly from $\inn{\mu}_\alpha$ defined in \eqref{mu_ball}, in that the supremum is taken over all $\rho>0$, rather than $\rho>1$.

In the $L^2$ setting, a variety of results for general fractal measures are known (see, e.g., \cite{Wolff2, Mattila, Oberlin, DuZhang}), but much less is understood beyond the $L^2$ framework.
For product measures, however, optimal results have been obtained by Lee--Lee--Roncal \cite{LLR},
with further developments connected to Assouad dimensions.
Suppose $\nu$ is supported on $[0,1]$ and satisfies
\[
\nu((t-\rho,t+\rho)) \le C\rho^\alpha, \quad (t,\rho)\in \R\times \R_{+}.
\]

In \cite[Theorem 1.6]{LLR}, it is shown that if $\gamma\ge -\frac \alpha 2$, then $H^\gamma$--$L_t^2(d\nu;L_x^2(-1,1))$ estimates hold for $e^{-it\partial_x^2} f$.
Moreover, optimal results in higher dimensions and weighted Strichartz-type estimates of the form $H^\gamma$--$L_t^q(d\nu;L_x^r)$ for fractional Schr\"odinger operators are established in \cite{LLR} (see also \cite{BRRS} for analogous results for the wave operator).

For the wave equation, the situation is better understood: not only are the product-type estimates optimal (see \cite{BRRS}), but there are also extensive results on $L^p-L^q$ estimates with respect to more general fractal measures (see \cite{Oberlin, CHL, IKSTU, HKL2, CDK}; see also \cite{Har, HKL} for related results in the case of product measures).

We consider the local smoothing type estimates for the Schr\"odinger operator relative to a fractal measure $\mu$:
\begin{align}\label{fls}
\big\| e^{-it\partial_x^2}f \big\|_{L^p(\R\times [0,1],\mu)}
\le C [\mu]_\alpha^{1/p} \|f\|_{L_\gamma^p(\R)}
\end{align}
for some $C=C_{\alpha,p,\gamma}$.
By combining the weighted square function estimates in \Cref{cor:sqweighted} with the standard strategy, we obtain the following.

\begin{theorem}\label{FLSS}
Let $0\le \alpha \le 2$, and let $\mu$ be a Borel measure on $\R^2$ such that $[\mu]_\alpha \le1$.
\vspace{-.2cm}
\begin{enumerate}[(i)]
\item
\textit{(Sufficiency)}
For $2\le p\le 4$, the estimate \eqref{fls} holds whenever
\[
  \gamma>\gamma(\alpha,p):=
  \begin{cases}
		\frac{2-\alpha}{2}, & \alpha\in[1,2],\\[.8ex]
\frac{2-\alpha}{2}+\frac{\alpha-1}{p}, & \alpha\in[0,1].
  \end{cases}
\]
\item
\textit{(Necessity)}
Conversely, \eqref{fls} can hold only if
\[
  \gamma\ge
  \begin{cases}
  \max\!\big(1-\frac{2\alpha}{p},\,\frac{2-\alpha}{p}\big), & \alpha\in[1,2],\\[.8ex]
\max\!\big(1-\frac{\alpha+1}{p},\,\frac{\alpha}{p}\big), & \alpha\in[0,1].
  \end{cases}
\]
\end{enumerate}
\end{theorem}
 The sufficient conditions $\gamma(\alpha,p)$ are essentially sharp for $p=4$ for all $\alpha\in[0,2]$ and $p=2$ for $\alpha\in [1,2]$.
We prove the sufficiency part of \Cref{FLSS} in \Cref{sec:ls} and the necessity part in \Cref{examp}.

\section{More examples}\label{sec:MoreExamples}
\subsection{A lower bound for the weighted square function estimate}\label{sec:BBCRV}

We give a lower bound for the weighted square function estimate by using an example from \cite{BBCRV0} discussed in \Cref{remark:lowerBoundSquare}.

We fix a parameter $0< \kappa\leq 1/2$. For each $l\in R^{-\kappa} \Z \cap [-1/2,1/2]$, let $\Omega_l = [l-R^{-1},l+R^{-1}]$ and
\[f_l(x,t) = \eta(R^{-1}t)\eta(R^{-1}x) R\int e^{i(x\xi+t\xi^2)}  1_{\Omega_l}(\xi)d\xi,\]
where $\eta$ is defined in \Cref{sec:etaDefn}.
It follows that $|f_l(x,t)| \sim 1$ on $B_R$ and decays rapidly away from $B_R$. Thus, 
\[ \big\| \big(\sum_l |f_l|^2 \big)^{1/2}\big\|_{L^p} \les R^{\kappa/2} R^{2/p}.   \]
Let $\Omega = \cup_l \Omega_l$ and $f= \sum_{l} f_l$. For a sufficiently small $0<c<1$, define 
\begin{equation}\label{eqn:Gamma}
   \Gamma = (2\pi R^\kappa \Z \times 2\pi R^{2\kappa} \Z ) \cap B_{cR}(0), \;\; Y = \Gamma + B_{c}(0). 
\end{equation}
One can check that $|Y\cap B_\rho| \les \rho^{2-3\kappa}$ for all $\rho \geq 1$. Therefore, if we let $\alpha = 2-3\kappa$, then $H=1_Y$ is an $\alpha$-dimensional weight. 
Moreover, we have $x\xi + t\xi^2 \in 2\pi \Z + B_{0.01}(0)$ whenever $(x,t)\in Y$ and $\xi \in \Omega$. Consequently, 
\[ \| f\|_{L^p(Y)} \sim R^\kappa |Y|^{1/p} \sim R^{\kappa+(2-3\kappa)/p}.\]
Combining these estimates, we get the lower bound
\begin{align}\label{ex_new}
\| f\|_{L^p(Y)} / \big\| \big(\sum_l |f_l|^2 \big)^{1/2}\big\|_{L^p} \gtrsim R^{\kappa(\frac{1}{2} - \frac{3}{p})} = R^{-(2-\alpha)(\frac{1}{p}-\frac{1}{6})}. 
\end{align}

\subsection{A sharp example beyond interpolation}\label{sec:exInterpolation}
In view of \Cref{cor:SchFreqLocal} concerned with $\alpha$-dimensional weights or measures, it seems natural to ask whether our weighted $L^p$-estimates, \Cref{thm:BR} and \Cref{thm:SchFreqLocal} expressed in terms of $\max_{U\in \U} \kappa_{p,H}(U)$, can yield results beyond what can be obtained by interpolating between the $L^2$ and $L^4$ estimates that they provide. The following example shows that the answer is affirmative. For this particular weight, the dominant scale $R^{-1/2}\leq s\leq 1$ for $\max_{U\in \U} \kappa_{p,H}(U)$ depends on the exponent $p$, being either 1 or $R^{-1/2}$.
\begin{ex}\label{ex:Y}
Let $1< \alpha < 2$ and $p_\alpha = 4/(3-\alpha)$. We construct a positive weight $H=1_Y$ in $\R^2$ such that $\inn{H}_\alpha \lesssim1$ and
\begin{equation}\label{eqn:kappaY}
		\sup_{U\in \U} \kappa_{p,H}(U)  
		\les \begin{cases}
		R^{-\frac{2-\alpha}{2p}}, &  2\leq p\leq p_\alpha\\[1ex]
		R^{-\frac{3-\alpha}2(\frac1p-\frac14)}, & p_\alpha \leq p\leq 4.
	\end{cases}
\end{equation}
\end{ex}

Note that for the weight in \Cref{ex:Y}, \Cref{thm:BR} and \Cref{thm:SchFreqLocal}  yield $L^p$-estimates which cannot be obtained by interpolating between $L^2$ and $L^4$ estimates for $2<p<4$.

Regarding \Cref{cor:sqweighted}, let $\sigma$ denote the infimum of exponents for which the following estimate holds with the specific weight $H=1_Y$ to be defined: 
\[ \|f\|_{L^p(Hdx)} \less R^\sigma \big\| \big(\sum_\theta |f_\theta|^2 \big)^{\frac12} \big\|_{L^p(\R^2)}. \]
Applying \eqref{eqn:kappaY} to \Cref{cor:sqweighted} yields that
\begin{equation}\label{eqn:zeta}
\sigma \le \max \Big\{ -\frac{2-\alpha}{2p},~ -\frac{3-\alpha}2\Big(\frac1p-\frac14\Big)\Big\}.
 \end{equation}
 Note that the upper bound \eqref{eqn:zeta} is strictly stronger than the one
 \begin{equation}\label{eqn:kappaAlpha1}
 	\sigma \le -(2-\alpha) \Big(\frac{1}{p}-\frac{1}{4}\Big)
 \end{equation}
given by \eqref{eqn:kappaAlpha} for all intermediate $2<p<4$.  

The upper bound \eqref{eqn:zeta} is  sharp for $2\leq p\leq p_\alpha$ and $p=4$. This can be seen from the lower bound
\begin{equation}\label{eqn:zetaLower}
\sigma \geq \max \Big\{ -\frac{2-\alpha}{2p},~ -2\big( \frac{1}{p}-\frac{1}{4} \big) \Big\}.
\end{equation}
The lower bound follows from examples in \Cref{remark:lowerBoundSquare}. Specifically, the unit ball example provides the lower bound $\sigma \geq -2\big( \frac{1}{p}-\frac{1}{4}\big)$, while the single wave packet example yields the lower bound $\sigma \geq  -\frac{2-\alpha}{2p}$ since $Y\subset T$ for a single $R^{1/2}\times R$ tube $T$ with $|Y|/|T| \sim R^{-(2-\alpha)/2}$. On the other hand, the computation from \Cref{sec:BBCRV} yields a lower bound weaker than \eqref{eqn:zetaLower}. 

We compare lower and upper bounds for $\sigma$ in \Cref{fig1}.

\begin{figure}
	\begin{tikzpicture}[scale=12]
		\pgfmathsetmacro{\aval}{3/2}
		\draw[->] (0,-4/20) -- (0,1/15) node[below left] {$\sigma$};
		\draw[->] (0,0) -- (11/20,0) node[right] {$\tfrac{1}{p}$};
		\draw (1/4,0) -- ++(0,0) node[above left] {$\tfrac14$};
		\draw (1/2,0) -- ++(0,0) node[above] {$\tfrac12$};
		
		\pgfmathsetmacro{\xint}{1/(\aval + 2)}
		\pgfmathsetmacro{\yint}{1/2 - 2*\xint}
		\pgfmathsetmacro{\xk}{(3 - \aval)/4}
		\pgfmathsetmacro{\yk}{-(3 - \aval)/2*(\xk - 1/4)}
		
		\draw[gray,densely dashed] (\xint,0) -- (\xint,\yint);
		\draw (\xint,0) -- ++(0,0)
		node[above] {$\tfrac{1}{\alpha+2}$};
		
		\draw[gray,densely dotted] (1/2,0) -- (1/2,-0.13);
		
		\draw[gray,dash dot] (\xk,0) -- (\xk,\yk);
		\draw (\xk,0) -- ++(0,0)
		node[above] {\textcolor{blue}{$\tfrac{3-\alpha}{4}$}};
		
		\begin{scope}
			\clip (0,-3/20) rectangle (11/20,1/20);
			
			\draw[thick,domain=1/4:1/2]
			plot (\x,{-(2-\aval)*(\x-1/4)});
			
			\draw[dotted,thick,domain=1/4:\xint]
			plot (\x,{1/2-2*\x});
			\draw[dotted,thick,domain=\xint:1/2]
			plot (\x,{-(2-\aval)/2*\x});
			
			\draw[blue,thick,domain=1/4:\xk]
			plot (\x,{-(3-\aval)/2*(\x-1/4)});
			\draw[blue,thick,domain=\xk:1/2]
			plot (\x,{-(2-\aval)/2*\x});
		\end{scope}
		
		\filldraw[fill=black,draw=black,thick]
		(\xint,\yint) circle (0.001);
		
		\begin{scope}
			\small
			\def\xL{0.70}
			\def\xR{0.78}
			\def\yTop{-0.010}
			\def\dy{0.045}
			
			\draw[blue,thick]
			(\xL,\yTop) -- (\xR,\yTop)
			node[right] {\footnotesize \eqref{eqn:zeta}};
			
			\draw[thick]
			(\xL,\yTop-\dy) -- (\xR,\yTop-\dy)
			node[right] {\footnotesize \eqref{eqn:kappaAlpha1}};
			
			\draw[dotted,thick]
			(\xL,\yTop-2*\dy) -- (\xR,\yTop-2*\dy)
			node[right] {\footnotesize \eqref{eqn:zetaLower}};
		\end{scope}
	\end{tikzpicture}
		\caption{Blue and black lines denote the upper bounds on $\sigma$ from \eqref{eqn:zeta} and \eqref{eqn:kappaAlpha1}, respectively, while the dotted line indicates the lower bound from \eqref{eqn:zetaLower}.
		\label{fig1}}
\end{figure}
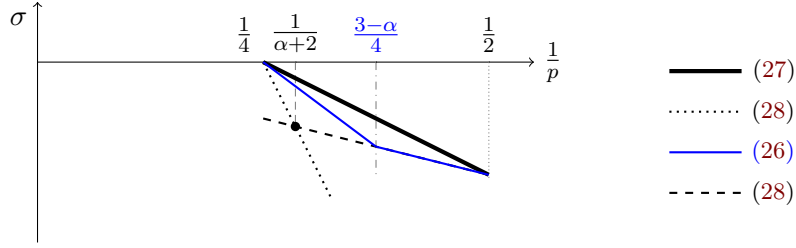

\begin{proof}[\textit{A weight $H$ satisfying \protect\eqref{eqn:kappaY}}]
Let $\kappa\in(0,\frac16)$ and $\alpha:=2-6\kappa\in(1,2)$. 
	Let  \[Y = \Gamma\cap ([0,R^{1/2}]\times [0,R]) + B_c(0),\]
    where $\Gamma$ is defined in \eqref{eqn:Gamma}. Let $H=1_Y$. For any $z\in \R^2$ and $\rho \ge1$, we have
\begin{align*}
&H(B_\rho(z)) = |Y\cap B_\rho(z)| \\
&\lesssim
\begin{cases}
1, & \rho \in [1, R^\kappa],\\[1ex]
\rho/R^\kappa, & \rho \in [R^\kappa, R^{2\kappa}],\\[1ex]
\rho^2/R^{3\kappa}, & \rho \in [R^{2\kappa}, R^{1/2}],
\end{cases}
\qquad
\begin{cases}
(R^{1/2}/R^\kappa)(\rho/R^{2\kappa}), & \rho \in [R^{1/2}, R],\\[1ex]
R^{3/2}/R^{3\kappa}, & \rho \in [R, \infty).
\end{cases}
\end{align*}

It follows that $|Y\cap B_\rho| \les \rho^{2-6\kappa}$ for all $\rho\ge1$.
Thus $H$ is $\alpha=(2-6\kappa)$-dimensional.

Next, we prove \eqref{eqn:kappaY} by computing  $H(T)$ and $H(U)$ directly.
Let $T\in \T_\tau$ and $U\in \U_\tau$ for some $|\tau|=s$. 
We have
\[
\frac{H(U)}{|U|} = \frac{|U\cap Y|}{|U|} \le \frac{|Y|}{|U|} \sim \frac{R^{\frac 32-3\kappa}}{R^2s}.
\]
Regarding $T$, we have 
\begin{align*}
\frac{H(T)}{|T|} = \frac{|T\cap Y|}{|T|} \lesssim
\begin{cases}
	s^{-3}R^{-3\kappa}/s^{-3}, & s\in [R^{-1/2},R^{-\kappa}], \\
	1/s^{-3}, & s \in [R^{-\kappa},1].
\end{cases}
\end{align*}
Therefore,
\begin{align*}
	\sup_{U\in \U} \kappa_{p,H}(U) &\les \max\Big( 
	\max_{R^{-\frac12}\le s\le R^{-\kappa}} R^{-\frac {3\kappa} 4} (s^{-1} R^{-(\frac 1 2 + 3\kappa )})^{\frac{1}{p}-\frac{1}{4} },
	\max_{R^{-\kappa}\le s\le 1} s^{\frac 3 4} (s^{-1} R^{-(\frac 1 2 + 3\kappa )})^{\frac{1}{p}-\frac{1}{4} } \Big).
	\end{align*}

Recall that $2\leq p\leq 4$. Substituting $\kappa=(2-\alpha)/6$ yields
\begin{align*}
	\sup_{U\in \U} \kappa_{p,H}(U)
		&\lesssim \max\Big(  R^{-\frac {3\kappa} p}, R^{-(\frac 1 2 + 3\kappa )(\frac{1}{p}-\frac{1}{4}) } \Big)
=\max\Big(R^{-\frac{2-\alpha}{2p}},~ R^{-\frac{3-\alpha}2(\frac1p-\frac14)}\Big),
\end{align*}
which is equivalent to \eqref{eqn:kappaY}.
\end{proof}

\section{Weighted wave envelope estimates}\label{sec:bil}
In this section, we prove \Cref{thm:sqweighted} using the bilinear restriction theorem.

\subsection{Decomposition}
Let $K\sim \log R$ be a dyadic number. Let $m\in \N$ be the integer such that $K^m \leq R^{1/2} < K^{m+1}$. We consider a canonical covering $\{\tau\}$ such that $|\tau|=s$ for each scale  
\[
s=K^{-j}, \quad j=0,1,\cdots,m.
\]
For each block $\tau$ at an intermediate scale, we define
\[f_{\tau} = \sum_{\theta \subset \tau} f_\theta.\]

We use the following pointwise estimate from \cite[Section 5]{DGW}.
\begin{lemma}\label{ptwise bound}
For any $x\in \R^2$, we have the pointwise bound
\begin{equation}\label{eqn:iteration}
  |f(x)|^p \less \sum_{|\theta|=R^{-1/2}} |f_{\theta}(x)|^{p} + \sum_{1\leq j\leq m} \sum_{|\tau|=K^{-j+1}} \sum_{\substack{\tau_1,\tau_2\subset \tau \\ |\tau_1|=|\tau_2| = K^{-j}  \\ d(\tau_1,\tau_2)\geq   K^{-j}}} |f_{\tau_1}(x)f_{\tau_2}(x)|^{\frac p2},
\end{equation}
where $d(\tau_1,\tau_2)$ denotes the distance between $\tau_1$ and $\tau_2$. 
\end{lemma}
The inequality follows from the Bourgain--Guth argument \cite{BG}. For the convenience of the reader, we provide a proof in Appendix \ref{sec:pt}.

\begin{remark}
	The $C_\eps R^\eps$ loss in our estimates arises from the broad--narrow decomposition in 
	\eqref{eqn:iteration}. Its proof produces a factor $C^m K^{O(1)}$ for an absolute constant $C>1$, whereas the other losses  in our argument are at most polylogarithmic in $R$. Thus, an efficient broad--narrow analysis of the type developed  by Guth, Maldague and Wang \cite{GMW} may allow the $C_\eps R^\eps$ loss to be replaced by a  polylogarithmic loss of the form $(\log R)^{O(1)}$; see  \cite[Section 2.1]{KWY} for such a broad--narrow decomposition. We do not pursue this refinement here.
\end{remark}

By integrating \eqref{eqn:iteration} with the weight $H$, we get 
\begin{align}\label{eqn:intg}
	\|f\|_{L^p(Hdx)}^p \less \sum_{\theta} \| f_{\theta} \|_{L^p(Hdx)}^p + \sum_{1\leq j\leq m} \sum_{|\tau|=K^{-j+1}} \max_{\substack{\tau_1,\tau_2\subset \tau \\ |\tau_1|=|\tau_2| = K^{-j}  \\ d(\tau_1,\tau_2)\geq  K^{-j}}}   \!\!\big\| |f_{\tau_1}|^{\frac12} |f_{\tau_2}|^{\frac12}\big\|_{L^p(Hdx)}^p,
\end{align}
where we have used the fact that, for each fixed $\tau$, there are at most $K^{O(1)} \less 1$ pairs $\tau_1, \tau_2 \subset \tau$. 

By the computation \eqref{eqn:exKnapp}, the first term on the right-hand side of \eqref{eqn:intg} is bounded by
\[
\sum_{\theta} \| f_{\theta} \|_{L^p(Hdx)}^p
\lesssim
\sum_{\theta} \sum_{U\in \U_\theta}\frac{H(U)}{|U|} |U|^{1-\frac{p}{2}} \| f_\theta\|_{L^2(w_U)}^p.
\]
This corresponds to the contribution at the smallest scale $s=R^{-1/2}$ in \eqref{eqn:envelope} (cf. \eqref{eqn:22}).
Since the computation in \eqref{eqn:exKnapp} relies on the assumption that $f_\theta$ is essentially constant on each $U\in\U_\theta$, we provide a rigorous justification below. By the Fourier support condition on $f_\theta$, we have 
\[  |f_\theta(x)|^2 \les |f_\theta|^2* \phi_{U_\theta}(x) \]
for a suitable $L^1$-normalized function $\phi_{U_\theta}\geq 0$ adapted to the tube $U_\theta\in \U_\theta$ centered at the origin. Then $|f_\theta|^2* \phi_{U_\theta}(x)$ is essentially constant on each $U\in \U_\theta$. Thus,
\begin{align*}
	\int_{U} |f_\theta|^p H &\les \int_{U} \bigl( |f_\theta|^2*\phi_{U_\theta} \bigr)^{p/2} H \\
	&\sim H(U) |U|^{-\frac{p}{2}}  \Bigg( \int_U |f_\theta|^2*\phi_{U_\theta} \Bigg)^{p/2} = \frac{H(U)}{|U|} |U|^{1-\frac{p}{2}} \| f_\theta\|_{L^2(w_U)}^{p},
\end{align*}
where $w_U (y) = \int_U  \phi_{U_\theta}(x-y) dx$ is an  $L^\infty$-normalized weight adapted to $U$. Summing this estimate  over $U\in\U_\theta$ and then over $\theta$ yields the claimed estimate.

Thus it remains to control the bilinear terms, which reduces to the following proposition.

\begin{prop}\label{prop:bil-lp}
Let $1\le K \ll R$ and $\tau_1,\tau_2\subset \tau$ satisfy $|\tau_j|=K^{-1}|\tau|$, $j=1,2$, and $d(\tau_1,\tau_2)\ge K^{-1}|\tau|$. Then for any $\epsilon>0$, $2\le p \le 4$ and $U\in \U_{\tau}$, we have
\begin{equation}\label{eqn:localbilinear}
       \big\| |f_{\tau_1}|^{\frac12} |f_{\tau_2}|^{\frac12}\big\|_{L^p(U,\, Hdx)}^p  \less  
        \,\kappa_{p,H}(U)^p |U|^{1-\frac{p}{2}} \big\| \big(\sum_{\theta \subset \tau} |f_{\theta}|^2 \big)^{1/2}  \big\|_{L^2(w_{U})}^p.
\end{equation}
\end{prop}

Summing \eqref{eqn:localbilinear} over all $U\in  \mathbb U_\tau$ yields
\[ 
   \big\| |f_{\tau_1}|^{\frac12} |f_{\tau_2}|^{\frac12}\big\|_{L^p(Hdx)}^p  \less \sum_{U\in \U_{\tau}} \kappa_{p,H}(U)^p |U|^{1-\frac{p}{2}} \big\| \big(\sum_{\theta \subset \tau} |f_{\theta}|^2 \big)^{1/2}  \big\|_{L^2(w_{U})}^p.
\]
This estimate bounds the bilinear terms on the right-hand side of \eqref{eqn:intg} by the right-hand side of \eqref{eqn:envelope}.

\subsection{Proof of \Cref{prop:bil-lp}}
We first reduce the proof of \eqref{eqn:localbilinear} to the case $p=4$,  which follows from the bilinear restriction theorem.
We claim that \Cref{prop:bil-lp} holds with $p=4$:
\begin{equation}\label{eqn:bilinear2w}
	\big\| |f_{\tau_1}|^{\frac12} |f_{\tau_2}|^{\frac12}\big\|_{L^4(U,\, Hdx)}  \less
	\kappa_{4,H}(U) |U|^{-\frac14} \big\| \big(\sum_{\theta \subset \tau} |f_{\theta}|^2 \big)^{1/2}  \big\|_{L^2(w_{U})}.
\end{equation}
Here, by definition \eqref{kappa_def}, when $p=4$ we have
\begin{align}\label{k4}
\kappa_{4,H}(U)^4=
\max_{\substack{T \in \T_{\tau(U)}:\\ T\subset U}} \frac{H(T)}{|T|}.
\end{align}

Having established \Cref{prop:bil-lp} for $p=4$,
it remains to treat the range $2\le p< 4$.
By H\"older's inequality,
\begin{equation}\label{eqn:111}
       \big\| |f_{\tau_1}|^{\frac12} |f_{\tau_2}|^{\frac12}\big\|_{L^p(U,\, Hdx)}
       \le H(U)^{\frac1p-\frac14} \big\| |f_{\tau_1}|^{\frac12} |f_{\tau_2}|^{\frac12}\big\|_{L^4(U,\, Hdx)}.
\end{equation}

We combine \eqref{eqn:bilinear2w} and \eqref{eqn:111}, and observe that
\[
\kappa_{4,H}(U)|U|^{-\frac14}H(U)^{\frac1p-\frac14}=\max_{\substack{T \in \T_{\tau(U)}:\\ T\subset U}} \Big(\frac{H(T)}{|T|}\Big)^{\frac14}\Big(\frac{H(U)}{|U|}\Big)^{\frac1p-\frac14}|U|^{\frac1p-\frac12}.
\]
This yields the desired estimate \eqref{eqn:localbilinear}, thereby completing the proof of \Cref{thm:sqweighted}. \qed

\subsection{Bilinear restriction estimates}
It remains to establish \eqref{eqn:bilinear2w}. To this end, we apply parabolic rescaling. Fix $\tau$ with $|\tau|=s$ centered at $(c,c^2)$. Let $A_\tau$ denote the affine transform 
\[ A_\tau (\xi_1,\xi_2) = (c,c^2) + (s\xi_1,2cs\xi_1+s^2 \xi_2). \]
Thus $\tau$ can be identified with the image of $[-1,1]\times [-2,2]$ under $A_\tau$. 
For each $\theta\subset \tau$, we define $g_\theta$ by 
\[ \hat{g_\theta}(\xi)=s^3 \hat{f_\theta}(A_\tau \xi) \]
and let 
\[
g_{\tau_i}:=\sum_{\theta \subset \tau_i} g_\theta, \quad i=1,2.
\]
With the new scale $R_s:=Rs^2$, $\widehat{g_\theta}$ is supported on a canonical box of dimensions $R_s^{-1/2}\times R_s^{-1}$ covering the $R_s^{-1}$-neighborhood of the parabola $\mathcal P$. Moreover, the supports of $\hat{g_{\tau_1}}, \hat{g_{\tau_2}}$ are $K^{-1}$-separated. 

We may write $g_\theta(x) = c_\tau(x) f_{\theta}(L_\tau x)$, where $|c_\tau(x)|=1$ and $L_\tau$ is the linear transform 
\begin{equation}\label{eqn:Ltau}
 L_\tau = \begin{pmatrix} s^{-1} & -2cs^{-2} \\ 0 & s^{-2}    
\end{pmatrix}.   
\end{equation}
Let $B=L_\tau^{-1}(U)$, which is a cube of side length $R_s$. 

Partition $B$ into unit cubes $q\subset B$. Recalling  \eqref{set of tube} and applying the change of variables $x\rightarrow L_\tau x$, we obtain 
\[
\kappa_{4,H}(U)^4= \max_{q\subset B} \frac{H(L_\tau(q))}{|L_\tau(q)|} =  \max_{q\subset B} \int_q H(L_\tau x) dx  
=\max_{q\subset B}  H_\tau(q), 
\]
where
\[ H_\tau (x)= H(L_\tau x).  \]

By rescaling,  for $w_U = w_B \circ L_{\tau}^{-1}$, 
\begin{equation}\label{eqn:rescaling}
	 \int_U |f_{\tau_1} f_{\tau_2}|^2 H = |\det L_\tau| \int_B  |g_{\tau_1} g_{\tau_2}|^2 H_\tau, \quad  \text{and} \quad 
	 \int \sum_{\theta \subset \tau} |f_\theta|^2 w_U  = |\det L_\tau|  \int \sum_{\theta \subset \tau} |g_\theta|^2 w_B.
\end{equation}
With this rescaling and using the identity $|B|^{-1}=|U|^{-1}|\det L_\tau|$, the estimate \eqref{eqn:bilinear2w} reduces to the following lemma.

\begin{lemma}\label{lem:l4}
Let $B$ and $g_{\tau_i}$ be as above and let $H_\tau$ be a weight. Then
\begin{equation}\label{eqn:sum2}
\int_{B} |g_{\tau_1}g_{\tau_2}|^2 H_\tau\, \less \max_{q\subset B} H_\tau(q) |B|^{-1}  \| g_{\tau_1} \|_{L^2(w_B)}^{2} \| g_{\tau_2}  \|_{L^2(w_B)}^{2}
\end{equation}
where $w_B$ is an $L^\infty$-normalized weight which decays rapidly away from $B$.
\end{lemma}
By the local $L^2$-orthogonality, for $\tau_1,\tau_2\subset \tau$ we have
\begin{equation}\label{eqn:localorthogonality}
	\| g_{\tau_i}  \|_{L^2(w_B)}  \les  
	\big\| \big(\sum_{\theta \subset \tau} |g_{\theta}|^2 \big)^{1/2} \big\|_{L^2(\widetilde w_B)}, \quad i=1,2,
\end{equation}
where $w_B$ and $\widetilde w_B$ are $L^\infty$-normalized weights adapted to $B$. We henceforth denote both by $w_B$.
Thus, combining \eqref{eqn:rescaling}, \eqref{eqn:sum2}, and \eqref{eqn:localorthogonality}, we obtain \eqref{eqn:bilinear2w}.

It remains to prove \Cref{lem:l4}. For this purpose, we invoke the following local bilinear restriction estimate for the parabola:

\begin{theorem}\label{thm:bil}
Let $1\le K\ll r$ and $B$ be a ball of radius $r$.
Let $\tau_1$ and $\tau_2$ be boxes contained in $N_{r^{-1}}(\cP)$ such that $d(\tau_1,\tau_2)\ge K^{-1}$. If $g_{\tau_1}$ and $g_{\tau_2}$ are Fourier supported on $\tau_1$ and $\tau_2$, respectively, then
\[
\int_B |g_{\tau_1}g_{\tau_2}|^2    \les K^{O(1)} |B|^{-1} \| g_{\tau_1} \|_{L^2(w_B)}^2 \| g_{\tau_2}  \|_{L^2(w_B)}^2,
\]
where $w_B$ is an $L^\infty$-normalized weight which decays rapidly away from $B$.
\end{theorem}
The proof of Theorem \ref{thm:bil} is standard; see, for example, \cite{TaoR} or \cite[Lemma 2.4]{SLee}.

We now turn to the proof of \Cref{lem:l4}.

\begin{proof}[Proof of \Cref{lem:l4}]
Since $|g_{\tau_1}|^2 |g_{\tau_2}|^2$ has compact Fourier support, it is locally constant on unit cubes. 
More precisely, for $\phi= (1+|\cdot |)^{-200}$, we have 
\[
|g_{\tau_1} g_{\tau_2}|^2\les |g_{\tau_1} g_{\tau_2}|^{2} * \phi.
\]
Note that the convolution $|g_{\tau_1} g_{\tau_2}|^{2} * \phi$ is locally constant on unit cubes in the sense that  $|g_{\tau_1} g_{\tau_2}|^{2} * \phi(x) \sim |g_{\tau_1} g_{\tau_2}|^{2} * \phi(y)$ whenever $|x-y|\les 1$. Hence, for each unit cube $q\subset B$,
\[
\int_{q} |g_{\tau_1}g_{\tau_2}|^2H_\tau
\les H_\tau(q)\int_{q} |g_{\tau_1} g_{\tau_2}|^{2} * \phi. \]
Summing this inequality over $q\subset B$ yields
\begin{equation}\label{eqn:sum1}
\int_{B} |g_{\tau_1}g_{\tau_2}|^2 H_\tau
\les \max_{q\subset B} H_\tau(q) \int_{B}  |g_{\tau_1} g_{\tau_2}|^{2} * \phi. 
\end{equation}

The integral on the right-hand side of \eqref{eqn:sum1} can be written as 
\[ \int |g_{\tau_1}g_{\tau_2}|^2 \phi*1_B.   \] 
By  the rapid decay of $\phi$, we have $\phi*1_B \les \sum_{z\in \Z^2} (1+|z|)^{-100} 1_{B_z}$, where $B_z=rz+B$ and $r$ is the side-length of the cube $B$. Applying  \Cref{thm:bil} to each translate $B_z$ and summing the resulting estimates over $z$ by Cauchy--Schwarz yields the desired bound in Lemma \ref{lem:l4}.
\end{proof}

\section{Fractal local smoothing estimates}\label{sec:ls}
In this section, we establish the sufficiency parts in \Cref{FLSS_par} and \Cref{FLSS}. 
\subsection{Rescaling and reductions}\label{sec:rescale}
We reduce \Cref{FLSS_par} and \Cref{FLSS} to \Cref{cor:SchFreqLocal}. By homogeneity, we may assume $[\mu]_{\alpha} = 1$ and $[\mu]_{\beta,\mathrm{par}} = 1$ for the proofs of \Cref{FLSS_par} and \Cref{FLSS}, respectively. 

Fix $2\leq p \leq 4$. Given a measure $\mu$, we need to verify
\[
	\big\| e^{-it\partial_x^2}f  \big\|_{L^p(\R\times[0,1];\mu)}
	\les  \|f\|_{L^p_\gamma(\R)}
\]
for all $\gamma>\gamma_{\mathrm{par}}(\beta)$ and $\gamma>\gamma(\alpha,p)$ respectively.  By a standard Littlewood--Paley reduction, it suffices to show that 
\[ 
\| e^{-it\partial_x^2}f \|_{L^p(\R\times[0,1];\mu)}
\less R^{\gamma} \|f\|_p
\]
for any function $f$ whose Fourier transform is supported on  $\{\xi\in\R: |\xi| \le R\}$.

For this purpose, we define $f_R(x) = f(R^{-1}x)$ so that $\hat{f_R}$ is supported on $[-1,1]$ and $\| f_R \|_p = R^{\frac1p}\|f\|_p$.  A change of variable $\xi\rightarrow R\xi$ gives
\begin{align*}
	|e^{-it\partial_x^2}f(x)| \sim |\mathbf{U}_{R^2} f_R (Rx, R^2 t)| , \;\; (x,t) \in \R \times [0,1],
\end{align*}
where $\mathbf{U}_{R^2}$ is defined in  \eqref{eqn:defnU_R}. 
We also define the rescaled measure $\mu_R$ on $\R^2$ by $\mu_R(E) = \int 1_E(Rx,R^2 t) d\mu(x,t)$ so that we have  
\[
	\int |\mathbf{U}_{R^2} f_R|^p(x,t)\,d\mu_R(x,t) = \int |\mathbf{U}_{R^2} f_R|^p (Rx, R^2 t)\,d\mu(x,t).
\]
Thus, we have 
\begin{equation}\label{eqn:reductionFLS}
    \big\| e^{-it\partial_x^2}f \big\|_{L^p(\R\times[0,1];\mu)}
\les \big\| \bU_{R^2} f_R \big\|_{L^p(\R^2; \mu_R)}.
\end{equation}

\begin{lemma}\label{lem:measure} For $\mu_R$ defined as above, we have 
\vspace{-.2cm}
	\begin{enumerate}
    	\item If $[\mu]_{\beta,\mathrm{par}}\le 1$, then 
		\begin{align*}
			\inn{\mu_R}_{(\beta+1)/2} &\les R^{-\beta},  \;\;\;  \beta \in [1,3], \\
			\inn{\mu_R}_{\beta} &\les R^{-\beta}, \;\;\; \beta\in[0,1].
		\end{align*}
		\item If $[\mu]_{\alpha}\le 1$, then 
		\[ \inn{\mu_R}_{\alpha}  \les   \begin{cases}
								R^{1-2\alpha},  & \alpha \in [1,2],\\
			R^{-\alpha},  & \alpha \in [0,1].
		\end{cases}. \]
	\end{enumerate}
\end{lemma}
\begin{proof}
Consider the ball $B_\rho(z)$ of radius $\rho\geq 1$ for $z=(z_1,z_2)$.
By the definition of $\mu_R$, we have  
\[
\mu_R\big(B_\rho(z)\big) \leq \mu\Big( \big(\frac{z_1-\rho}R,\frac{z_1+\rho}{R}\big)\times \big(\frac{z_2-\rho}{R^2},\frac{z_2+\rho}{R^2}\big)
\Big).
\]

Suppose that $[\mu]_{\beta,\mathrm{par}}=1$. 
We may cover the rectangle 
$\big(\frac{z_1-\rho}R,\frac{z_1+\rho}{R}\big)\times \big(\frac{z_2-\rho}{R^2},\frac{z_2+\rho}{R^2}\big)$
by $O(\rho^{1/2})$ parabolic rectangles of dimensions $\rho^{1/2}/R \times \rho/R^2$. Alternatively, we can just cover it by a parabolic rectangle of dimensions $\rho/R \times (\rho/R)^2$. Therefore, we have 
\[ 
   \mu_R\big(B_\rho(z)\big) \les \min\big( \rho^{\frac12} \big(\rho^{\frac12}/R\big)^{\beta}, (\rho/R)^\beta \big) = R^{-\beta} \min( \rho^{\frac{1+\beta}2},\rho^\beta),
\]
which verifies the claim. 

Next, assume that $[\mu]_{\alpha}=1$. We may cover the rectangle $\big(\frac{z_1-\rho}R,\frac{z_1+\rho}{R}\big)\times \big(\frac{z_2-\rho}{R^2},\frac{z_2+\rho}{R^2}\big)$
by a single ball of radius $\sim \rho/R$, or alternatively, $O(R)$ balls of radius $\rho/R^2$. Thus, 
\[ 
   \mu_R\big(B_\rho(z)\big) \les \min\big( R \big(\rho/R^2\big)^\alpha , \big(\rho/R\big)^\alpha\big)  = \min\big( R^{1-2\alpha}, R^{-\alpha}\big) \rho^\alpha. \qedhere
\]
\end{proof}

\subsection{Proof of sufficient conditions in \Cref{FLSS_par} and \Cref{FLSS}}
We prove sufficient conditions in \Cref{FLSS_par} and \Cref{FLSS}. 
Let $f$ be a function whose Fourier transform is supported on  $[-R,R]$, $R\geq 1$. 
By the reduction in \eqref{eqn:reductionFLS}, we have
\begin{align}\label{newtheta}
\| e^{-it\partial_x^2}f \|_{L^p(\R\times[0,1];\mu)} \les R^{-\frac{\theta}{p}} \| \bU_{R^2} f_R \|_{L^p(\R^2, R^\theta \mu_R)}
\end{align}
where $\theta\in \R$ is to be chosen depending on the measure $\mu$.

\subsubsection{Proof of \Cref{FLSS_par}}
Suppose that $[\mu]_{\beta,\mathrm{par}}= 1$.
By \Cref{lem:measure}, we have $\inn{R^\beta \mu_R}_{(\beta+1)/2} \les 1$ when $\beta\in[1,3]$ and $\inn{R^\beta \mu_R}_{\beta} \les 1$ when $\beta\in[0,1]$.
Thus, \Cref{cor:SchFreqLocal} yields
\begin{align*} 
\big\| \bU_{R^2} f_R \big\|_{L^p(\R^2, R^\beta \mu_R)} 
\less 
\begin{cases}
R^{2\left(\frac{1}{p}-(2-\frac{\beta+1}{2})\left(\frac{1}{p}-\frac{1}{4}\right)\right)} \| f_R\|_{L^p}
=R^{\frac{3-\beta}4+\frac{\beta-1}p}\|f_R\|_{L^p}, &\quad \beta\in [1,3],\\
R^{2\left(\frac{1}{p}-(2-\beta)\left(\frac{1}{p}-\frac{1}{4}\right)\right)} \| f_R\|_{L^p}
=R^{\frac{2-\beta}2-\frac{2(1-\beta)}p}\|f_R\|_{L^p}, &\quad \beta\in[0,1].
\end{cases}
\end{align*}
Recalling \eqref{newtheta} with $\theta=\beta$ and $\| f_R\|_{p} = R^{1/p} \|f\|_p$, we get 
\begin{align*}
\| e^{-it\partial_x^2}f \|_{L^p(\R\times[0,1];\mu)} 
\less
\begin{cases}
R^{\frac{3-\beta}{4}} \| f\|_{L^p}, \quad & \beta\in[1,3],\\[.8ex]
R^{\frac{2-\beta}{2} - \frac{1-\beta}{p}} \| f\|_{L^p}, \quad & \beta\in[0,1].
\end{cases}
\end{align*}
This completes the proof of the sufficiency part in \Cref{FLSS_par}. \qed

\subsubsection{Proof of \Cref{FLSS}}

The proof of the sufficient condition in \Cref{FLSS} for $\alpha\in [0,1]$ is the same as the proof for the case $\beta \in [0,1]$ of \Cref{FLSS_par}. For $\alpha\in [1,2]$, by \Cref{lem:measure}, we have $\inn{R^{2\alpha-1} \mu_R}_{\alpha} \les 1$. Thus, \Cref{cor:SchFreqLocal} yields
  \[ \big\| \bU_{R^2} f_R \big\|_{L^p(\R^2, R^{2\alpha-1} \mu_R)} \less R^{2\left(\frac{1}{p}-(2-\alpha)\left(\frac{1}{p}-\frac{1}{4}\right)\right)} \| f_R\|_{L^p}.   \]
     Consequently, by \eqref{newtheta} we have  
    \[  \big\| e^{-it\partial_x^2}f \big\|_{L^p(\R\times[0,1];\mu)} \les R^{-\frac{2\alpha-1}{p}} \| \bU_{R^2} f_R \|_{L^p(\R^2,R^{2\alpha-1} \mu_R)}  \less R^{\frac{2-\alpha}{2}-\frac1p} \| f_R\|_{L^p}
    =R^{\frac{2-\alpha}{2}} \| f\|_{L^p}. \]
This completes the proof. \qed

\subsection{A sketch of the proof of \eqref{sqfc}} \label{sec:sqfc}
We restate \eqref{sqfc} with $R$ replaced by $R^2$:
\begin{prop}\label{prop:sqfc}
   For $2\le p \le 4$, we have
\begin{align}\label{eq0}
\Big\| \big(\sum_J |\mathbf U_{R^2} f_J|^2\big)^{1/2} \Big\|_{L^p}
\less R^{\frac2p}\|f\|_{L^p},
\end{align}
whenever $\hat{f}$ is supported in $[-1,1]$.
\end{prop}
A proof of \Cref{prop:sqfc} is given in \cite{Yung}. We sketch the proof following the note. 

Let $\psi$ be a smooth function supported on $[-1,1]$ such that $\sum_{k\in \Z}\psi(\cdot+k)=1$ on $\R$. Consider the partition of $\R$ by intervals $J$ of length $R^{-1}$ centered at $c_J \in R^{-1} \Z$. 
Let $\psi_J(\xi)=\psi(R(\xi-c_J))$ so that $\sum_{J} \psi_J = 1$. We then decompose $f=\sum_J f_J$, where $\widehat f_J=\widehat f \psi_J$.

Fix $\tilde \psi \in C_c^\infty(\R)$  such that $\tilde \psi \psi=\psi$.
We may write $\mathbf U_{R^2} f_J(x,t)=K_J^t\ast f_J(x)$, where 
\[
K_J^t(x) = (2\pi)^{-1}\eta(R^{-2}t)\int e^{ix\xi+it\xi^2} \tilde \psi(2R(\xi-c_J))\,d\xi.
\]
By changing variables $\xi\rightarrow R^{-1}\xi+c_J$,
we observe that
\[
|K_J^t(x)|=(2\pi R)^{-1}\Big|\eta(R^{-2}t)\int e^{iR^{-1}(x+2tc_J)\xi+iR^{-2}t\xi^2}\tilde\psi(2\xi)\,d\xi\Big|.
\]
Integration by parts yields the decay estimate
\begin{align}\label{decay-ker}
|K_J^t(x)|
\le C_N R^{-1}(1+R^{-1}|x+2tc_J|+R^{-2}|t|)^{-N}, \quad N\ge1,
\end{align}
which, in particular, implies $\|K_J^t\|_{L_x^1} \lesssim 1$.
Consequently, by the Cauchy-Schwarz inequality, we have
$
   |\mathbf U_{R^2} f_J(x,t)|^2\lesssim  |K_J^t|*|f_J|^2(x). 
$
For $q=(p/2)'$, let $g \in L^{q}(\R^2)$ with $\|g\|_{L^{q}}=1$. Thus we obtain
\[ \int \sum_J |\bU_{R^2}f_J(x,t)|^2 g(x,t)\,dxdt \les \int \sum_J |f_J(y)|^2 \mathfrak Mg (y) \,dy, \]
where $\mathfrak Mg$ is defined by
\[
\mathfrak Mg (y) = \sup_J \int |K_J^t(x-y) g(x,t)|\,dxdt, \quad y\in \R.
\]

By duality and H\"older, we find that 
\[ \Big\| \big(\sum_J |\mathbf U_{R^2} f_J|^2\big)^{1/2} \Big\|_{L^p}^{2}
\lesssim
\big\|\big(\sum_{J} |f_J|^2 \big)^{1/2}\big\|_p^2
\,\,\big\| 
\mathfrak M 
\big\|_{L^{q} \to L^q}.\]
By the one-dimensional analogue of \eqref{LP}, for $2\le p \le \infty$,
$
\big\| \big(\sum_J |f_J|^2 \big)^{1/2}\big\|_{L^p(\R)}\lesssim \|f\|_{L^p(\R)}.
$
Thus in order to establish \eqref{eq0}, it remains to prove that for $2\le p \le 4$, 
\begin{align}\label{Mgpp}
\big\|\mathfrak M\big\|_{L^{q} \to L^q}
\less R^{2\cdot\frac 2p}, \;\; q=\Big(\frac p2\Big)'.
\end{align}

To show this, for $w\in [-1,1]$ we set
$
T_w = \{(x,t)\in \R^2: |x+2tw| \le R^{-1}, ~|t|\le 1\},
$
and define the Nikodym maximal function by
\[
\mathfrak Ng(y):= \sup_{w \in [-1,1]} \frac{1}{|T_w|}\int_{T_w+(y,0)} |g(x,t)|\,dxdt, \quad y\in\R.
\]
The Nikodym maximal function satisfies the following bounds; see e.g. \cite{MSS, Yung}.
\begin{prop}
For $2 \le q \le \infty$ and $R\ge1$, we have
\[
\| \mathfrak Ng\|_{L^q(\R)} \less \|g\|_{L^q(\R^2)}.
\]
\end{prop}

In view of \eqref{decay-ker}, the operator $\mathfrak M$ can be dominated by a rescaled version of $\mathfrak N$. Indeed, after the rescaling $y\rightarrow R^2y$ and $(x,t)\rightarrow (R^2x,R^2t)$ and noting $|T_w|=R^{-1}$, we have
\[
\|\mathfrak M\|_{q\rightarrow q} \lesssim R^2 R^{-\frac2q}\|\mathfrak N\|_{q\rightarrow q} \less R^{\frac{2}{q'}} = R^{\frac{4}{p}},
\]
which verifies \eqref{Mgpp}.

\section{Examples for lower bounds}\label{examp}
\newcommand{\Mtheta}{[\mu]_{\theta}}

In this section, we  discuss the lower bounds for the regularity $\zeta$ and $\gamma$ in \Cref{cor:SchFreqLocal}, \Cref{FLSS_par}, and \Cref{FLSS}.

\subsection{Lower bounds for $\zeta$ in \Cref{cor:SchFreqLocal}}\label{sec:global_low}
We show that \eqref{global_LS} holds only if
\begin{align}\label{gamma0}
\zeta \ge \zeta(\alpha,p):=\max\Big(\frac12-\frac1p, ~\min\Big(\frac{\alpha}{2p}, \frac{2\alpha-1}{2p}\Big)\Big).
\end{align}

\textit{$(i)$ Proof of $\zeta\ge \frac12-\frac1p$.}
To show this, we use an example from \cite{Rogers}. Let $\psi\in C_c^\infty([1/4,4])$ be such that $\psi=1$ on $[1/2,2]$. We take $\widehat f(\xi)=e^{-iR\xi^2}\psi(\xi)$ so that
\[
f(x)=(2\pi)^{-1}\int e^{ix\xi-iR\xi^2}\psi(\xi)\,d\xi.
\]
By integration by parts, we obtain $|f(x)|\lesssim (R+|x|)^{-N}$ for $|x|\ge 10R$
while the stationary phase method gives $|f(x)|\lesssim R^{-1/2}$ for $|x|\le 10R$. Hence 
$\|f\|_{L^p} \lesssim R^{-\frac12+\frac1p}.$

Note that
\[
|\mathbf U_Rf(x,t)|\gtrsim  \Big|\int e^{ix\xi+i(t-R)\xi^2}\psi(\xi)\,d\xi\Big|.
\]
In particular, one has $|\mathbf U_Rf(x,t)|\gtrsim 1$ on the set 
\[ F=\{(x,t)\in \R^2:|x|\le c,~|t-R|\le c\}	 \]
for some small constant $c>0$. Thus
\begin{align*}
\mu(F)^{\frac 1p}\les \big\| \mathbf U_Rf \big\|_{L^p(\R\times [0,R],\mu)}
\les \inn{\mu}_\alpha^{\frac1p} R^{\zeta}R^{-\frac12+\frac{1}{p}}
\end{align*}
follows. 
If we take $\mu=1_F(x,t)dxdt$, then $\inn{\mu}_\alpha \le1$ for any $\alpha \in [0,2]$. Since $\mu(F)\sim1$, we obtain $\zeta \ge \frac12-\frac1p$ as desired.

\textit{$(ii)$ Proof of $\zeta\ge \min\big(\frac{\alpha}{2p}, \frac{2\alpha-1}{2p}\big)$.}
For a smooth function $\psi$ as above, choose $g$ such that $\widehat{g}(\xi)=\psi(R^{\frac12}(\xi+1))$.
Then $\|g\|_{L^p} \lesssim R^{-\frac12+\frac1{2p}}$.

By the change of variable $\xi\rightarrow R^{-\frac12}\xi-1$, ignoring the extra oscillatory factor independent of $\xi$, we obtain
\[
\big|\mathbf U_Rg(x,t)\big| \sim \Big|\int e^{ix\xi+it\xi^2}\psi(R^{\frac12}(\xi+1))\,d\xi\Big|
=R^{-\frac12}\Big|\int e^{iR^{-1/2}(x-2t)\xi+itR^{-1}\xi^2} \psi(\xi)\,d\xi\Big|.
\]
Thus $|\mathbf U_Rg(x,t)|\gtrsim R^{-\frac12}$ for $(x,t)\in G$ where 
\begin{align}\label{setG}
G=\{(x,t): |x-2t|\le cR^{1/2}, ~ |t|\le R\}	
\end{align}
for a small constant $c>0$. Thus
\begin{align}\label{GGG}
R^{-\frac12}\mu(G)^{\frac1p}\lesssim
\big\|e^{-it\partial_x^2}g\big\|_{L^p(\R\times [0,R],\mu)}
\lesssim \inn{\mu}_\alpha^{1/p} R^{\zeta-\frac12+\frac1{2p}}.
\end{align}

For a set $G$ given in \eqref{setG}, cover $G$ by union of $O(R^{\frac12})$ disjoint balls of radius $R^{\frac12}$. When $\alpha\in[1,2]$ we obtain
\begin{align*}
|G\cap B_\rho(z)|
\lesssim \begin{cases}
\rho^2 \le R^{\frac{2-\alpha}2}\rho^\alpha, \quad~ & \rho \in [1,R^{1/2}],\\[.8ex]
R^{\frac12}\rho \le R^{\frac{2-\alpha}2}\rho^\alpha, \quad~ & \rho \in [R^{1/2},\infty).
\end{cases}
\end{align*}
When $\alpha \in [0,1]$, we have
\[
|G\cap B_\rho(z)|
\lesssim
\begin{cases}
\rho^2 \le R^{\frac{2-\alpha}2}\rho^\alpha \le R^{\frac32-\alpha}\rho^\alpha, \quad~ & \rho \in [1,R^{1/2}],\\[.8ex]
R^{\frac12}\rho \le R^{\frac32-\alpha}\rho^{\alpha}, \quad~ & \rho \in [R^{1/2},R],\\[.8ex]
R^{\frac32} \le R^{\frac32-\alpha}\rho^{\alpha}, \quad~ & \rho \in [R,\infty).
\end{cases}
\]
We now take $\mu=\min\big(R^{\frac{\alpha-2}2},R^{\alpha-\frac32}\big)1_G(x,t)dxdt$. Then $\inn{\mu}_\alpha\le1$ follows.
Since $|G|\sim R^{\frac32}$, we have $\mu(G)\sim \min\big(R^{\frac{\alpha+1}2},R^{\alpha}\big)$.
Thus \eqref{GGG} yields the lower bounds $\zeta \ge \min\big(\frac{\alpha}{2p}, \frac{2\alpha-1}{2p}\big)$.

\subsection{Lower bounds for \Cref{FLSS_par}}
To get lower bounds for $\gamma$ in \Cref{FLSS_par}, we rescale the estimates in \Cref{sec:global_low} by following the argument in \Cref{sec:rescale}.

Suppose
\begin{equation}\label{fls-theta}
  \big\| e^{-it\partial_x^2}f \big\|_{L^p(\R\times[0,1],\,d\mu)}
  \;\le\; C\, [\mu]_{\beta,\mathrm{par}}^{1/p}\, R^\gamma\|f\|_{L^p(\R)}
\end{equation}
holds for all positive measures $\mu$ satisfying $[\mu]_{\beta,\mathrm{par}} \le1$
where $\widehat {f}$ is Fourier supported on $[R/2,R]$.
For a given $f$, we set $f_R =f(R^{-1}\cdot)$ so that $\widehat {f_R}$ is supported on $[1/2,1]$ and $\|f_R\|_p =R^{\frac1{p}}\|f\|_{p}$.
Similarly as before in \Cref{sec:rescale}, we define a positive measure $\mu_R$ by
$
\mu_R(E) = \int 1_E(Rx,R^2 t) d\mu(x,t)
$
for any $E\subset \R^2$.
Then by \eqref{eqn:reductionFLS}, we have
\[
\big\|e^{-it\partial_x^2}f\big\|_{L^p(\R\times [0,1],\mu)} \sim 
\big\|\mathbf U_{R^2}f_R\big\|_{L^p(\R\times [0,R^2],\mu_R)}.
\]

Next, for a given positive Borel measure $\mu$ satisfying $[\mu]_{\beta,\mathrm{par}}\le1$, \Cref{lem:measure} provides
\[
\inn{R^\beta\mu_R}_{\theta} \les 1
\]
where $\theta=\frac{\beta+1}2$ when $\beta\in[1,3]$, and $\theta=\beta$ when $\beta\in[0,1]$.
By applying \eqref{fls-theta}, we have
\begin{align}\label{hRh}
\big\|\mathbf U_{R^2} f_R\big\|_{L^p(\R\times [0,R^2],R^\beta\mu_R)}
\lesssim R^{\frac{\beta}p+\gamma}\|f\|_{L^p}
=R^{\frac{\beta}p+\gamma-\frac{1}{p}}\|f_R\|_{L^p}. 
\end{align}

The discussion in \Cref{sec:global_low} shows that if \eqref{hRh} holds, then $\frac \beta p +\gamma -\frac 1p\ge2\zeta(\theta,p)$ which is defined by \eqref{gamma0}. Equivalently, \eqref{hRh} holds only if
\begin{align*}
\gamma \ge 
\begin{cases}
2\max\big(\frac12-\frac1p, ~\frac{\theta}{2p}\big) +\frac{1-\beta}p, \quad~ &\theta\in[1,2],\\[1ex]
2\max\big(\frac12-\frac1p, ~\frac{2\theta-1}{2p}\big) +\frac{1-\beta}p, \quad~&\theta \in [0,1]
\end{cases}
\end{align*}
where $\theta=\frac{\beta+1}2$ when $\beta\in[1,3]$, and $\theta=\beta$ when $\beta\in[0,1]$.
Hence,
\[
  \gamma\ge
  \begin{cases}
  \max\!\big\{1-\frac{\beta+1}{p},\, \frac{3-\beta}{2p}\big\}, & \beta\in[1,3],\\[1ex]
\max\!\big\{1-\frac{\beta+1}{p},\,\frac\beta p\}, & \beta\in[0,1].
  \end{cases}
\]

\subsection{Proof of \Cref{FLSS}}
We now discuss the lower bounds for $\gamma$ of \Cref{FLSS}.
Suppose
\begin{equation}\label{fls-theta2}
  \big\| e^{-it\partial_x^2}f \big\|_{L^p(\R\times[0,1],\,d\mu)}
  \;\le\; C\, [\mu]_{\alpha}^{1/p}\, R^\gamma\|f\|_{L^p(\R)}
\end{equation}
holds. Let $\mu$ be a positive measure satisfying $[\mu]_\alpha \le1$ for $\alpha\in[0,2]$.
By \Cref{lem:measure}, we have
\[
\inn{R^\theta \mu_R}_\alpha \lesssim 1
\]
for $\theta=2\alpha-1$ (when $\alpha\in[1,2]$) and $\theta=\alpha$ (when $\alpha \in [0,1]$).
As before, if \eqref{fls-theta2} holds, then
\begin{align}\label{hRh2}
\big\|\mathbf U_{R^2} f_R\big\|_{L^p(\R\times [0,R^2],R^\theta\mu_R)}
\lesssim R^{\frac{\theta}p+\gamma}\|f\|_{L^p}
=R^{\frac{\theta}p+\gamma-\frac{1}{p}}\|f_R\|_{L^p}. 
\end{align}
The discussion in \Cref{sec:global_low} shows that if \eqref{hRh} holds, then $\frac \theta p+\gamma-\frac1p\ge 2\zeta(\alpha,p)$. Equivalently, we have
\begin{align*}
\gamma \ge 
\begin{cases}
2\max\big(\frac12-\frac1p, ~\frac{\alpha}{2p}\big) +\frac{1-(2\alpha-1)}p, \quad~ &\alpha\in[1,2],\\[1ex]
2\max\big(\frac12-\frac1p, ~\frac{2\alpha-1}{2p}\big) +\frac{1-\alpha}p, \quad~&\alpha \in [0,1].
\end{cases}
\end{align*}
This yields
\[
  \gamma\ge
  \begin{cases}
  \max\!\big(1-\frac{2\alpha}{p},\,\frac{2-\alpha}{p}\big), & \alpha\in[1,2],\\[.8ex]
\max\!\big(1-\frac{\alpha+1}{p},\,\frac{\alpha}{p}\big), & \alpha\in[0,1].
  \end{cases}
\]

\appendix

\section{Proof of \Cref{ptwise bound}}\label{sec:pt}
In this section, we prove \Cref{ptwise bound}, motivated by the argument in \cite{BG}.
The following elementary lemma will be used repeatedly.

\begin{lemma}\label{lem:BG} Let $\{ a_i \}_{i\in I}$ be a sequence of non-negative real numbers indexed by a finite set $I$. For each $i\in I$, let $I_i \subset I$ be a subset containing $i$ such that $|I_i| \leq C_1$ for all $i\in I$ for some 
constant $C_1 \in \N$. Then there exists $C=C(C_1,p)$ such that for $p\ge1$,
\[ 
\Big(\sum_{i\in I} a_i\Big)^p \leq C\Big( \max_{i\in I} a_i^p + (\# I )^p\max_{\substack{i\in I,\\ j\notin I_i}} a_i^{\frac p2} a_j^{\frac p2}\Big). \]
\end{lemma}
\begin{proof} 
Let $*, **\in I$ be the indices for which $a_* = \max_{i\in I} a_i$ and $a_{**}=\max_{j\notin I_{*}} a_{j}$. Then, for any $i\notin I_{*}$ we have $a_i \leq a_{*}^{1/2} a_{**}^{1/2}$. Therefore,
\begin{align*}
\sum_{i \in I} a_i &= \sum_{i\in I_{*}} a_i + \sum_{i\notin I_{*}} a_i \leq C_1 a_{*} + (\# I) a_{*}^{1/2} a_{**}^{1/2}.
\end{align*}
Taking $p$-th power and using the inequality $(x+y)^p\les_p x^p+y^p$ for $x,y>0$, we obtain the desired bound.
\end{proof}

Now we prove \Cref{ptwise bound}.

\begin{proof}[Proof of \Cref{ptwise bound}]
Recall that $K\sim \log R$ is a dyadic number and $m$ is the integer such that 
\[1\le K\le \dots \le K^m \leq R^{1/2} < K^{m+1}.\]
Let $\mathcal T_0 =\{ \tau_0\}$ denote the collection of the unit cube covering the parabola $\mathcal P$.

At the first stage, we decompose $\tau_0$ into a collection $\mathcal T_1=\mathcal T_1(\tau_0)$ of $K^{-1}\times K^{-2}$ boxes $\tau_1$ covering the $K^{-2}$-nbd of the parabola.
For each $\tau_1\in \mathcal T_1$,  let $\mathcal T_2(\tau_1)$ be a  collection of $K^{-2}\times K^{-4}$ boxes $\tau_2 \subset \tau_1$ covering $K^{-4}$-nbd of $\mathcal P$, and define $\mathcal T_2 = \cup_{\tau_1\in \mathcal T_1} \mathcal T_2(\tau_1)$.

Proceeding inductively, for $2\le j \le m$, we define $\mathcal T_j(\tau_{j-1})$ as the collection of boxes $\tau_j \subset \tau_{j-1}$ of dimension $K^{-j}\times K^{-2j}$ covering $K^{-2j}$-neighborhood of $\mathcal P$ and $\mathcal T_j$ similarly.
Finally, denote $\mathcal T_m=\{\theta\}$ and for each $j=1,\dots,m$ set
\[f_{\tau_j} = \sum_{\theta \subset \tau_j} f_\theta.\]

For each $\tau_1\in \mathcal T_1$, let 
\[
\mathcal N_1(\tau_1)=\{\tau_1'\in \mathcal T_1:~ \tau_1' \cap 2\tau_1 \neq \emptyset\}.
\]
It is clear that $\mathcal N_1(\tau_1)$ consists of only $O(1)$  elements, and if $\tau_1' \notin \mathcal N_1(\tau_1)$ then $d(\tau_1,\tau_1')\ge 1/K$.
Applying \Cref{lem:BG} with the neighborhoods $\mathcal N_1(\tau_1)$, and using $\#\mathcal T_1\les K$, we obtain
 \[ |f(x)|^p 
 \leq C\max_{\tau_1 \in \mathcal T_1} |f_{\tau_1}(x)|^p + CK^p \max_{\substack{\tau_1,\tau_1' \in \mathcal T_1; \\ d(\tau_1,\tau_1')\ge K^{-1} }} |f_{\tau_1}(x)|^{\frac p2} |f_{\tau_1'}(x)|^{\frac p2}   \]
for some absolute constant $C$.

Applying \Cref{lem:BG} again to the first term $|f_{\tau_1}|^p=|\sum_{\tau_2 \in \mathcal T_2(\tau_1)}f_{\tau_2}|^p$, we get 
\begin{align*}
  |f(x)|^p &\leq C^2 \max_{\tau_1} \max_{\tau_2\in \mathcal T_2(\tau_1)} |f_{\tau_2}(x)|^p + C^2K^p \max_{\tau_1} \max_{\substack{\tau_2,\tau_2' \in \mathcal T_2(\tau_1);\\ d(\tau_2,\tau_2')\ge K^{-2} }} |f_{\tau_2}(x)|^{\frac p2} |f_{\tau_2'}(x)|^{\frac p2} \\ &\qquad \quad+  CK^p \max_{\substack{\tau_1,\tau_1' \in \mathcal T_1;\\ d(\tau_1,\tau_1')\ge K^{-1} }} |f_{\tau_1}(x)|^{\frac p2} |f_{\tau_1'}(x)|^{\frac p2}.
\end{align*}
Continuing in this manner,  we get 
\begin{equation*}
	\begin{split}
  |f(x)|^p &\leq C^m \max_{\tau_m\in \mathcal{T}_m} |f_{\tau_m}(x)|^p + C^mK^p \sum_{j=1}^m \sum_{\tau_{j-1} \in \mathcal T_{j-1}} \max_{\substack{\tau_j,\tau_j' \in \mathcal T_j(\tau_{j-1});\\ d(\tau_j,\tau_j')\ge K^{-j} }} |f_{\tau_j}(x)|^{\frac p2} |f_{\tau_j'}(x)|^{\frac p2} \\
   &\less \sum_{|\theta|=R^{-1/2}} |f_{\theta}(x)|^{p} +  \sum_{j=1}^m \sum_{\tau_{j-1} \in \mathcal T_{j-1}} \sum_{\substack{\tau_j,\tau_j' \in \mathcal T_j(\tau_{j-1});\\ d(\tau_j,\tau_j')\ge K^{-j}  }} |f_{\tau_j}(x)|^{\frac p2} |f_{\tau_j'}(x)|^{\frac p2},
\end{split}
\end{equation*}
where we have used $C^m K^{O(1)} \less 1$ together with the fact that each $\tau_m \in \mathcal{T}_m$ contains $O(K)$ elements $\theta$ of size $|\theta| = R^{-1/2}$.
This completes the proof of Lemma \ref{ptwise bound}.
\end{proof}

\subsection*{Acknowledgement} 

J. Kim  was supported in part by grants from the Research Grants Council of the Hong Kong Administrative Region, China (Project No. CityU 21309222 and CityU 11308924). H. Ko was supported in part by the NRF (Republic of Korea) grant RS-2024-00339824, G-LAMP RS-2024-00443714,  JBNU research funds for newly appointed professors in 2024, and RS-2026-25471967. 
We thank the organizers of the Special Topic School “Maximal Operators and Applications” and the Hausdorff Center for Mathematics, where our collaboration began.
We also thank Sanghyuk Lee  for helpful discussions. The second author would like to thank Seheon Ham  for previous collaboration related to this topic. 
We are grateful to Tony Carbery for bringing the references \cite{CarberySeeger, CarberySlide} to our attention and for pointing out the overlap between their results and the case $p=2$ of \Cref{cor:sqweighted}. 
We acknowledge the use of AI tools to assist in revising  figures and identifying minor errors in an earlier draft.

\end{document}